\documentclass[10pt, twoside, reqno]{article}
\usepackage{amssymb, amsmath,amsthm,amscd,latexsym,enumerate}

\newtheorem{thm}{Theorem}[section]
\newtheorem{theorem}[thm]{Theorem}
\newtheorem{definition}[thm]{Definition}

\newtheorem{remark}[thm]{Remark}
\newtheorem{question}[thm]{Question}
\newtheorem{lemma}[thm]{Lemma}
\newtheorem{corollary}[thm]{Corollary}

\newtheorem{discussion}[thm]{Discussion}
\newtheorem{historical background}[thm]{Historical Background}
\newtheorem{conventions-notation-definition}[thm]{Conventions, Notation, Definitions}

\newtheorem{remarks}[thm]{Remarks}
\newtheorem{notation}[thm]{Notation}

\numberwithin{equation}{section}

\def\cc {{\mathfrak c}}

\def\CC {{\mathbb C}}

\def\MM {{\mathbb M}}

\def\VV {{\mathbb V}}

\def\sA {{\mathcal A}}
\def\sB {{\mathcal B}}
\def\sC {{\mathcal C}}

\def\sP {{\mathcal P}}

\def\sT {{\mathcal T}}
\def\sU {{\mathcal U}}
\def\sV {{\mathcal V}}

\def\cf {\mathrm{cf}}

\def\dom {\mathrm{dom}}

\def\< {{\langle}}
\def\> {{\rangle}}

\def\cf {\mathrm{cf}}

\def\cf {\mathrm{cf}}
\def\two {\mbox{\bf 2}}

\begin{document}

\title{Cardinal invariants for $\kappa$-box products}
\author{W. W. Comfort\\
Department of Mathematics and Computer Science,\\
Wesleyan University,\\ 
Middletown, CT 06459, USA\\
E-mail: wcomfort@wesleyan.edu
\and
Ivan S. Gotchev$^1$\\
Department of Mathematical Sciences,\\
Central Connecticut State University,\\
New Britain, CT 06050, USA\\
E-mail: gotchevi@ccsu.edu}

\date{}

\maketitle


\renewcommand{\thefootnote}{}



\footnote{$^1$ The author expresses his gratitude to Wesleyan University for 
hospitality and support during his sabbatical in the spring semester, 2008.}

\renewcommand{\thefootnote}{\arabic{footnote}}
\setcounter{footnote}{0}

\begin{abstract}
The symbol $(X_I)_\kappa$ (with $\kappa\geq\omega$) denotes 
the space $X_I:=\Pi_{i\in I}\,X_i$ with the
$\kappa$-box topology; this has as base all sets of the form
$U=\Pi_{i\in I}\,U_i$ with $U_i$ open in $X_i$ and with
$|\{i\in I:U_i\neq X_i\}|<\kappa$. The symbols $w$, $d$ and $S$ denote
respectively weight, density character, and Souslin number. Generalizing
familiar results for the usual product space (the case
$\kappa=\omega$), the authors show {\it inter alia}:

{\bf Theorem~\ref{T2.9}(b).} If $\kappa\leq\alpha^+$, $|I|=\alpha$
and each $X_i$ contains the discrete space $\{0,1\}$ and satisfies
$w(X_i)\leq\alpha$, then $w((X_I)_\kappa)=\alpha^{<\kappa}$.

{\bf Theorem~\ref{HMPkappa}.} 
If $\kappa\leq\beta\leq 2^\alpha$ and
$X=(D(\alpha))^\beta$ with $D(\alpha)$ discrete, $|D(\alpha)|=\alpha$,
then $d((X_I)_\kappa)=\alpha^{<\kappa}$.

{\bf Theorem~\ref{T4.15}.}
Let $\alpha\geq3$ and $\kappa\geq\omega$ be cardinals, and let
$\{X_i:i\in I\}$ be a set of spaces such that $|I|^+\geq\kappa$ and  
$d(X_i)\leq\alpha\leq S(X_i)$ for each $i\in I$. Consider these conditions: {\rm (i)}~$\alpha$ is
regular; {\rm (ii)}~$\alpha=\alpha^{<\kappa}$; {\rm (iii)}
$\kappa\ll\alpha$; {\rm (iv)}~$S((X_J)_\kappa)=\alpha$ for all nonempty
$J\in[I]^{<\kappa}$. Then:

{\rm (a)} if conditions {\rm (i), (ii), (iii)} and {\rm (iv)} hold, then
$S((X_I)_\kappa)=\alpha^{<\kappa}=\alpha$; and

{\rm (b)} if one (or more) of conditions {\rm (i), (ii), (iii)} or
{\rm (iv)} fails, then
$S((X_I)_\kappa)=(\alpha^{<\kappa})^+$.

{\bf Corollaries \ref{C4.37}(a) and \ref{C4.40}.}
Let $\alpha\geq3$ and $\kappa\geq\omega$ be cardinals, and let
$\{X_i:i\in I\}$ be a set of spaces such that $|I|^+\geq\kappa$. 

{\rm (a)} If $\alpha^+\ge\kappa$ and $\alpha\le S(X_i)\le \alpha^+$ for each 
$i\in I$ then $\alpha^{<\kappa}\le S((X_I)_\kappa)\le(2^\alpha)^+$; and

{\rm (b)} if $\alpha^+\leq \kappa$ and $3\le S(X_i)\le \alpha^+$ for each 
$i\in I$ then $S((X_I)_\kappa)=(2^{<\kappa})^+$.
\end{abstract}

\section{Historical context}

The most prominent, most useful, and most-studied cardinal invariants
associated with topological spaces are the weight, density character,
and Souslin number. Countless papers and monographs over the decades
have given estimates, in some cases even precise evaluations, of the
value of these invariants for the usual Tychonoff product
$X_I=\Pi_{i\in I}\,X_i$ of a set of spaces $(X_i)_{i\in I}$ in terms of the
values for the initial spaces $X_i$. But in the case of $\kappa$-box
topologies (defined in Section~1 below) on spaces of the form $X_I$,
very little is known, and that is fragmentary and nowhere systematically assembled.

In this paper we study with considerable thoroughness those three cardinal
invariants for these modified box products, in each case seeking (as usual)
estimates for the product in terms of the values for the initial spaces.
Our methods are largely topological and set-theoretic, although as
expected certain computations are made precise only when ZFC is
enhanced with appropriate additional (consistent)  axioms.

Our work draws upon, and in some cases extends, published theorems of
R.~Engelking and M.~Kar\l owicz \cite{EnKa},
W.~W. Comfort and S.~Negrepontis \cite{comfneg72},
\cite{comfneg74}, \cite{comfneg82},
F.~S.~Cater, P.~Erd{\H{o}}s and F.~Galvin \cite{ceg},
W.~W. Comfort and L.~C. Robertson \cite{comfrobii} and 
M.~Gitik and S.~Shelah \cite{gitikshelah}.
We give details at appropriate points in the paper.

We acknowledge with thanks helpful e-mail correspondence from (a)~Ist\-v\'an
Juh\'asz, (b)~Stevo Todor\v{c}evi\'c, and (c) Santi Spadaro.

\section{Introduction}

Hypothesized topological spaces
here are not subjected to standing separation properties.
Special hypotheses are imposed locally, as required. 

$\alpha$, $\beta$, $\gamma$ and $\lambda$ are cardinals,
$\kappa$ is an infinite cardinal, $\omega$ is the least infinite cardinal, 
and $\cc$ is the cardinality of the interval $[0,1]$. As usual, for 
$\alpha\geq\omega$ we write $\alpha^+:=\min\{\beta:\beta>\alpha\}$. 

$\eta$ and $\xi$ are ordinals.

The symbols $w(X)$ and $d(X)$ denote respectively the weight and density
character of the space $X$. A \emph{cellular} 
family in a space $X$ is a family of pairwise disjoint nonempty open
subsets of $X$ and $S(X)$, the {\it Souslin number} of $X$, is the
cardinal number 
\begin{equation}\nonumber
\min\{\lambda~\mbox{: no cellular family}~\sA~\mbox{in}~X~\mbox{satisfies}~|\sA|=\lambda\}.
\end{equation}

We here follow many authors \cite{comfneg74}, \cite{comfneg82},
though not all \cite{E1}, \cite{juh71}, \cite{juh}, in allowing
$w$, $d$ and $S$ to assume finite values. If for example $X$ is a discrete
space of cardinality $17$, then $w(X)=d(X)=17$ and $S(X)=17^+=18$.

For $I$ a set, we write $[I]^\lambda:=\{J\subseteq I:|J|=\lambda\}$; the 
notations $[I]^{<\lambda}$, $[I]^{\leq\lambda}$ are defined analogously.
It is clear that if $\lambda>|I|$ then (a)~$[I]^\lambda=\emptyset$ and
(b)~$[I]^{<\lambda}$ is the full power set $\sP(I)$.

For a set $\{X_i:i\in I\}$ of sets we write
$X_I:=\Pi_{i\in I}\,X_i$. 
For $A=\Pi_{i\in I}\,A_i\subseteq X_I$ the
{\it restriction set of} $A$ is the set $R(A):=\{i\in I:A_i\neq X_i\}$.
When each $X_i=(X_i,\sT_i)$ is a space, we use
the symbol $(X_I)_\kappa$
to denote $X_I$ with the $\kappa$-{\it box topology}; this is
the topology for which the set
$$\sU:=\{\Pi_{i\in I}\,U_i:U_i\in\sT_i,|R(U)|<\kappa\}$$
is a base. (The $\omega$-box topology on $X_I$, then,
is the usual product topology.) We refer to $\sU$ as {\it the canonical
base} for $(X_I)_\kappa$, and to the elements of $\sU$ as {\it canonical
open sets}. By way of caution to the reader,
we note that even when $\kappa$ is regular, the
intersection of fewer than $\kappa$-many sets, each open in
$(X_I)_\kappa$, may fail to be open in $(X_I)_\kappa$. (Indeed each
space $X_i$ embeds homeomorphically as a (closed) subspace of
$(X_I)_\kappa$, so if some $X_i$ lacks that intersection property then
so does $(X_I)_\kappa$.)

For simplicity we denote by the symbol $\two$ the discrete space of
cardinality $2$, and for cardinals $\alpha\geq2$
we denote by $D(\alpha)$ the discrete space of cardinality $\alpha$.

For spaces $X$ and $Y$, the symbol $Y=_h X$ means that $Y$ and $X$ are
homeomorphic; the symbol
$Y\subseteq_h X$ means that $X$ contains a 
homeomorphic copy of the space $Y$.
\begin{definition}\label{DefSLC}
{\rm A cardinal $\kappa$ is a {\it strong limit cardinal} if
$\lambda<\kappa\Rightarrow2^\lambda<\kappa$.
}
\end{definition}
In \ref{beth}--\ref{expon} we cite the basic tools and facts we
need from the elementary theory of cardinal arithmetic. For motivation,
discussion and proofs where appropriate, see
\cite[\S1]{comfneg74}, \cite[Appendix~A]{comfneg82} or \cite{jech}.

The familiar {\it beth} cardinals $\beth_\xi(\alpha)$  are defined
recursively as follows.

\begin{definition}\label{beth}
{\rm
Let $\alpha\geq2$ be a cardinal. Then

(a) $\beth_0(\alpha):=\alpha$;

(b) $\beth_{\xi+1}(\alpha):=2^{\beth_\xi(\alpha)}$ for each ordinal
$\xi$; and

(c) $\beth_\xi(\alpha):=\Sigma_{\eta<\xi}\,2^{\beth_\eta(\alpha)}$ for
limit ordinals $\xi>0$.
}
\end{definition}

\begin{remarks}\label{bethprops}
{\rm
Let $\xi$ be a limit ordinal and let
$\alpha\geq2$ and $\lambda\geq\omega$ be cardinals. Then

(a) a set $S\subseteq\xi$ is cofinal in $\xi$ if and only if
$\{\beth_\eta(\alpha):\eta\in S\}$ is cofinal in $\beth_\xi(\alpha)$;
hence

(b) $\cf(\beth_\lambda(\alpha))=\cf(\lambda)$.
}
\end{remarks}

\begin{definition}\label{deflog}
{\rm
For $\alpha\geq\omega$, $\log(\alpha)$ is the cardinal number
$$\log(\alpha):=\min\{\beta:2^\beta\geq\alpha\}.$$
}
\end{definition}

\begin{notation}\label{defweakexp}
{\rm
Let $\kappa\ge\omega$ and $\alpha\ge 2$. Then
$\alpha^{<\kappa}:=\Sigma_{\lambda<\kappa}\,\alpha^\lambda$.
}
\end{notation}
It is well known and easy to prove that
$|[\alpha]^{\lambda}|=\alpha^\lambda$ when $\lambda\leq\alpha$, so
$|[\alpha]^{<\kappa}|=\Sigma_{\lambda<\kappa}\,\alpha^\lambda$ when
$\kappa\leq\alpha^+$. For ease of reference later, we build some
redundancy into the statement of Theorem~\ref{expon}.

\begin{theorem}\label{expon}
Let $\alpha\geq2$ and $\kappa\geq\omega$. Then

{\rm (a)} $\kappa\leq2^{<\kappa}\leq\alpha^{<\kappa}$;

{\rm (b)} if $\kappa$ is regular then
$\alpha^{<\kappa}=(\alpha^{<\kappa})^{<\kappa}$;

{\rm (c)} if $\kappa$ is singular,
then $(\alpha^{<\kappa})^{<\kappa}=\alpha{^\kappa}$; 

{\rm (d)} $((\alpha^{<\kappa})^{<\kappa})^{<\kappa}=(\alpha^{<\kappa})^{<\kappa}$.
\end{theorem}

\begin{remark}
{\rm
It is clear that the useful relation given in part (d) of
Theorem~\ref{expon}
is immediate from parts (b) and (c).
The authors are not acquainted with other examples in mathematics of
operators which, as in Theorem~\ref{expon}(d),
first stabilize at the third iteration. Responding to a request from one
of us (speaking in a seminar) for terminology
suitable for this phenomenon, Peter
Johnstone promptly proposed the expression ``sesquipotent".
}
\end{remark}

The condition
$(\alpha^{<\kappa})^{<\kappa}=\alpha^{<\kappa}$, satisfied by many pairs
of cardinals $\alpha$ and $\kappa$, will play a role frequently in this
paper. An alternate characterization is often useful.

\begin{theorem}\label{weak<kappa}
Let $\alpha\geq2$ and $\kappa\geq\omega$. Then 
\begin{itemize}
\item[{\rm (a)}] these conditions are equivalent:
\begin{itemize}
\item[{\rm (i)}] $\alpha^{<\kappa}=(\alpha^{<\kappa})^{<\kappa}$; and
\item[{\rm (ii)}] either $\kappa$ is regular or there is $\nu<\kappa$ such that
$\alpha^\nu=\alpha^{<\kappa}$.
\end{itemize}
\item[{\rm (b)}] If the conditions in (a) fail, then $\kappa$ and 
$\alpha^{<\kappa}$ are singular cardinals and 
$\cf(\alpha^{<\kappa})=\cf(\kappa)$.
\end{itemize}
\end{theorem}
\begin{proof} (a) ((i)~$\Rightarrow$~(ii)). If (ii) fails then $\kappa$ is a
limit cardinal and for every $\nu<\kappa$ there is a cardinal
$\lambda<\kappa$ such that $\alpha^\nu<\alpha^\lambda$, so also
$\alpha^{<\kappa}$ is a limit cardinal. It is easily checked that
$\cf(\kappa)=\cf(\alpha^{<\kappa})$, so
$$(\alpha^{<\kappa})^{<\kappa}\geq(\alpha^{<\kappa})^{\cf(\kappa)}
=(\alpha^{<\kappa})^{\cf(\alpha^{<\kappa}))}>\alpha^{<\kappa}.$$

((ii)~$\Rightarrow$~(i)). If $\kappa$ is regular we have
$(\alpha^{<\kappa})^{<\kappa}=\alpha^{<\kappa}$
by Theorem~\ref{expon}(b). Suppose then that $\kappa$ is singular, hence a
limit cardinal, and that there is $\nu<\kappa$ such that
$\alpha^\nu=\alpha^{<\kappa}$. Then
$$\begin{array}
{r@{=}l}
(\alpha^{<\kappa})^{<\kappa}&(\alpha^\nu)^{<\kappa}
=\Sigma_{\lambda<\kappa}\,(\alpha^\nu)^\lambda
=\Sigma_{\nu<\lambda<\kappa}\,(\alpha^\nu)^\lambda\\
&\Sigma_{\nu<\lambda<\kappa}\,\alpha^\lambda
=\alpha^{<\kappa}.
\end{array}$$

(b) Clearly $\kappa$ is singular, $\alpha^{<\kappa}$ is limit, and 
$$\cf(\alpha^{<\kappa})=\cf(\kappa)<\kappa\le \alpha^{<\kappa}.$$
Hence $\alpha^{<\kappa}$ is singular. 
\end{proof}

\begin{remarks}\label{R1.2}
{\rm
(a) As our title and our Abstract indicate,
we are concerned here with the weight, density character, and Souslin number 
of (sometimes specialized) products of the form
$(X_I)_\kappa$; the corresponding results are contained in 
Sections 2, 3 and 4, respectively.

(b) As the reader knows well, the ``functions" $w$, $d$ and $S$
enjoy specific useful monotonicity properties; we have in mind these
familiar phenomena:

\begin{itemize}
\item[({\it i})] If $X$ and $Y$ are spaces and $Y\subseteq_h X$,
then $w(Y)\leq w(X)$;
\item[({\it ii})] If $\sT_1$ and $\sT_2$ are topologies on a set $X$ with
$\sT_1\subseteq\sT_2$, then $d(X,\sT_1)\leq d(X,\sT_2)$ and
$S(X,\sT_1)\leq S(X,\sT_2)$.
\end{itemize}

On the other hand, both the analogue of ({\it i}) for $d$ and $S$ and of
({\it ii}) for $w$ can fail. For example, with $X=\two^\cc$ and
$$Y:=\{x\in X:|\{i\in I:x_i\neq0\}|=1\},$$
one has $Y$ discrete in $X$ with $|Y|=\cc$ and $d(X)=\omega<\cc=d(Y)$,
also $S(X)=\omega^+<\cc^+=S(Y)$.
And with $X=\two^\cc$ and $Y'$ a countable dense subset of $X$ one has
$w(Y')=w(X)=\cc$ when the usual product topology $\sT_1$ is considered,
but $w(Y',\sT_2)=\omega<\cc$ when $Y'$ is given the discrete topology
$\sT_2\supseteq\sT_1$. 

We use the indicated monotonicity properties ({\it i}) and ({\it ii})
frequently in this paper, without warning or comment. We use also the
fact that if $X$ is a space and $Y$ is dense in $X$, then
necessarily $S(Y)=S(X)$.}
\end{remarks}

\section{On the weight of $\kappa$-box products}

\begin{discussion}\label{oldweight}
{\rm
It is well known \cite[2.3.F(a)]{E1} for each set $\{X_i:i\in I\}$ of
$T_1$-spaces with $w(X_i)\geq2$ that $X_I:=\Pi_{i\in I}\,X_i$ satisfies
$$w(X_I)=\max\{\sup_{i\in I}\,w(X_i),|I|\}.$$
In particular,
\begin{equation}\label{Eq1}
w(X_I)=|I|~\mbox{if each}~X_i~\mbox{satisfies}~w(X_i)\leq|I|.
\end{equation}

In Theorem~\ref{T2.9} we give the
correct analogue of (\ref{Eq1}) for $\kappa$-box
topologies. 
}
\end{discussion}

\begin{lemma}\label{LC}
Let $\alpha\geq\omega$ and $\kappa\geq\omega$. Then
$$w((\two^\alpha)_\kappa)\ge\alpha.$$
\end{lemma}

\begin{proof}
Let $Y\subseteq X=\two^\alpha$ be as in Remarks~\ref{R1.2}(b). Then $Y$
is discrete in $\two^\alpha$, hence is discrete in $(\two^\alpha)_\kappa$,
so $$w((\two^\alpha)_\kappa)\ge w(Y)=\alpha.$$
\vskip-18pt
\end{proof}

\begin{theorem}\label{T1}
Let $\alpha\geq\omega$ and $\kappa\geq\omega$
and let $\{X_i:i\in I\}$ be a set of spaces
such that $w(X_i)\le\alpha$ for each $i\in I$. 
Then

{\rm (a)} $\sup_{i\in I}\,w(X_i)\le
w((X_I)_\kappa)\le\alpha^{<\kappa}\cdot|I|^{<\kappa}$; and

{\rm (b)} if in addition $\two\subseteq_h X_i$ for each $i\in I$, then
also $w((X_I)_\kappa)\geq|I|$.
\end{theorem}

\begin{proof}
(a) Let $\sB_i$ be a base for $X_i$ with $|\sB_i|\leq\alpha$ and with
$X_i\in\sB_i$, and for $\lambda<\kappa$ let
$$\sB(\lambda):=\{B=\Pi_{i\in I}\,B_i:B_i\in\sB_i,|R(B)|=\lambda\}.$$
Then $|\sB(\lambda)|\leq|[I]|^\lambda\cdot\alpha^\lambda$, and since
$\sB:=\bigcup_{\lambda<\kappa}\,\sB(\lambda)$ is a base for $(X_I)_\kappa$
we have
$$w((X_I)_\kappa)\leq|\sB|\leq\Sigma_{\lambda<\kappa}\,|[I]^\lambda|\cdot\alpha^\lambda
=|I|^{<\kappa}\cdot\alpha^{<\kappa}.$$
Since $X_i\subseteq_h (X_I)_\kappa$, 
we have $w((X_I)_\kappa)\ge w(X_i)$ for each $i\in I$.
Hence $w((X_I)_\kappa)\ge \sup_{i\in I}w(X_i)$. 

(b) It follows from $\two\subseteq_h X_i$ that $\two^I\subseteq_h X$ and 
from Lemma \ref{LC} we have $w((X_I)_\kappa)\ge w((\two^I)_\kappa)\ge |I|.$
\end{proof}

For future reference we re-state this portion of Theorem~\ref{T1}.

\begin{corollary}\label{C2.4}
Let $\alpha$ and $\kappa$ be infinite cardinals
and let $\{X_i:i\in I\}$ be a set
of spaces such that $|I|\leq\alpha$
and $w(X_i)\leq\alpha$ for each
$i\in I$. Then
$w((X_I)_\kappa)\leq\alpha^{<\kappa}$.
\end{corollary}

\begin{discussion}\label{D2.5}
{\rm
If $\omega\leq\alpha<\alpha^+<\kappa$, then
$w((\two^\alpha)_\kappa)=2^\alpha$, while
$\alpha^{<\kappa}\geq\alpha^{(\alpha^+)}=2^{(\alpha^+)}$. In many
models of set theory and for many cardinals $\alpha$ one has
$2^{(\alpha^+)}>2^\alpha$, and in such cases the inequality
$w((\two^\alpha)_\kappa)\leq\alpha^{<\kappa}$ of Corollary~\ref{C2.4}
becomes strict.
That explains why the formula $w(X_\kappa)=\alpha^{<\kappa}$ cannot
be asserted without restraint in Corollary~\ref{C2.4}, even
when $|I|=\alpha$. Our next goal in
this section is to show that, subject only to the simple restrictions
$\kappa\leq\alpha^+$ and $|I|=\alpha$, the inequality
$w((X_I)_\kappa)\leq\alpha^{<\kappa}$ of Corollary~\ref{C2.4}
becomes an equality (Theorem~\ref{T2.9}).
}
\end{discussion}

\begin{lemma}\label{L14}
Let $\alpha$ and $\kappa$ be infinite cardinals such that
$\kappa\leq\alpha^+$. Then

{\rm (a)} 
if $\lambda<\kappa$ and $\lambda\leq\alpha$,
then $w((\two^\alpha)_\kappa)\ge 2^\lambda$; and

{\rm (b)} $w((\two^\alpha)_\kappa)\geq2^{<\kappa}$.
\end{lemma}

\begin{proof}
(a) If $\kappa=\alpha^+$ then $(\two^\alpha)_\kappa$ is the discrete space
$D(2^\alpha)$, which has weight $2^\alpha=2^{<\kappa}\geq2^\lambda$. We
assume therefore that $\kappa\leq\alpha$.
The space $(\two^\lambda)_{\kappa}$ is then homeomorphic to a discrete subspace of
$(\two^\alpha)_\kappa$, so 
$w((\two^\alpha)_\kappa)\geq w((\two^\lambda)_{\kappa})=|\two^\lambda|=2^\lambda$.

(b) is immediate from (a).
\end{proof}

\begin{theorem}\label{T2.7}
Let $\kappa$ and $\alpha$ be infinite cardinals. Then

{\rm (a)} if $\kappa\geq\alpha^+$ then
$w((\two^\alpha)_\kappa)=2^\alpha$;

{\rm (b)} if $\kappa\leq\alpha^+$ then
$w((\two^\alpha)_\kappa)=\alpha^{<\kappa}$.
\end{theorem}

\begin{proof}
(a) is obvious, since $(\two^\alpha)_\kappa$ is discrete.

(b) The inequality $\leq$ is given by Corollary~\ref{C2.4}. We
show $\geq$.

If $\kappa=\alpha^+$ then $(\two^\alpha)_\kappa$ is the discrete space
$D(2^\alpha)$, which has weight
$2^\alpha=\alpha^\alpha=\alpha^{<\kappa}$. We assume in what follows
that $\kappa\leq\alpha$ and we consider two cases.

\underline{Case 1}. $2^{<\kappa}\leq\alpha$.

If $w((\two^\alpha)_\kappa)\geq\alpha^{<\kappa}$ fails then there is
$\lambda<\kappa$ such that $w((\two^\alpha)_\kappa)<\alpha^{\lambda}$; we
fix such
$\lambda$ and we choose in $(\two^\alpha)_\kappa$
a base $\sB$ of canonical open sets such that $|\sB|<\alpha^{\lambda}$.
From Lemma~\ref{L14}(b) we have $|\sB|\geq2^{<\kappa}$.

For every $A\in[\alpha]^{<\kappa}$ there
is $B\in\sB$ such that $A\subseteq R(B)$. (To check that, it is enough
to choose in $(\two^\alpha)_\kappa$ a
canonical open set $U=\Pi_{i\in I}\,U_i$ with $R(U)=A$ and $x\in U$,
and then to find $B\in\sB$ such that $x\in B\subseteq U$. Then $B$ is as
required.) Thus
\begin{equation}\label{Eq2}
[\alpha]^{<\kappa}\subseteq\bigcup\{[R(B)]^{<\kappa}:B\in\sB\}.
\end{equation}
For each $B\in\sB$ we have $|R(B)|<\kappa$ and hence
$[R(B)]^{<\kappa}=\sP(R(B))$. Therefore $|[R(B)]^{<\kappa}|=2^{|R(B)|}\leq2^{<\kappa}$.
From (\ref{Eq2}), then, we have the contradiction
$$\alpha^{<\kappa}=|[\alpha]^{<\kappa}|\leq\Sigma\{[R(B)]^{<\kappa}|:B\in\sB\}\leq$$
$$2^{<\kappa}\cdot|\sB|=|\sB|<\alpha^{\lambda}\leq\alpha^{<\kappa}.$$

\underline{Case 2}. Case 1 fails. 

Then there is $\lambda<\kappa$ such
that $2^\lambda>\alpha$. If the desired inequality
$w((\two^\alpha)_\kappa)\geq\alpha^{<\kappa}$ fails then there is
$\mu<\kappa$ such that  $w((\two^\alpha)_\kappa)<\alpha^\mu$, and
then with $\delta:=\max\{\lambda,\mu\}$ we have the contradiction
$$w((\two^\alpha)_\kappa)<\alpha^\delta\leq(2^\delta)^\delta=2^\delta=w((\two^\delta)_\kappa)\leq
w((\two^\alpha)_\kappa).$$
\vskip-20pt
\end{proof}

\begin{remark}
{\rm
When $\kappa=\alpha^+$ in Theorem~\ref{T2.7}, parts (a) and (b) are
compatible since
$\alpha^{<\kappa}=\alpha^\alpha=2^\alpha$ in that case.
}
\end{remark}

\begin{corollary}\label{C2.8}
Let $\kappa$ and $\alpha$ be infinite cardinals. Then
$$w((\two^{(\alpha^{<\kappa})})_\kappa)=(\alpha^{<\kappa})^{<\kappa}=
w((\two^{((\alpha^{<\kappa})^{<\kappa})})_\kappa).$$
\end{corollary}
\begin{proof}
From Theorem~\ref{expon}(a) we have
$\kappa\leq\alpha^{<\kappa}\leq(\alpha^{<\kappa})^{<\kappa}$.
The first equality then results by replacing $\alpha$ by
$\alpha^{<\kappa}$ in 
Theorem~\ref{T2.7}(b), the second equality results by making the same
substitution one more time.
\end{proof}

\begin{theorem}\label{T2.9}
Let $\alpha$ and $\kappa$ be infinite cardinals, and let
$\{X_i:i\in I\}$ be a set of spaces
such that $|I|=\alpha$, and $\two\subseteq_h X_i$ and
$w(X_i)\leq\alpha$ for each
$i\in I$. Then

{\rm (a)} if $\kappa\leq\alpha^+$ then $w((X_I)_\kappa)=\alpha^{<\kappa}$;

{\rm (b)} if $\kappa\geq\alpha^+$ then $w((X_I)_\kappa)=2^\alpha$.
\end{theorem}

\begin{proof}
We have $\two^\alpha\subseteq_h X$ and hence
$(\two^\alpha)_\kappa\subseteq_h (X_I)_\kappa$. Then from
Theorem~\ref{T2.7} and Corollary~\ref{C2.4} it follows that
$$\alpha^{<\kappa}=w((\two^\alpha)_\kappa)\leq
w((X_I)_\kappa)\leq\alpha^{<\kappa}$$ 
in (a), and
$$2^\alpha=w((\two^\alpha)_\kappa)\leq
w((X_I)_\kappa)\leq\alpha^{<\kappa}\leq\alpha^\alpha=2^\alpha$$
in (b).
\end{proof}

\begin{corollary}\label{wXIkappa}
Let $\alpha$ and $\kappa$ be infinite cardinals, and let 
$\{X_i:i\in I\}$ be a set of spaces such that $\two\subseteq_h X_i$ for
each $i\in I$. If
$w(X_i)\leq(\alpha^{<\kappa})^{<\kappa}$ for each $i\in I$ and
$\alpha^{<\kappa}\leq|I|\leq(\alpha^{<\kappa})^{<\kappa}$, then
$w((X_I)_\kappa)=(\alpha^{<\kappa})^{<\kappa}$.
\end{corollary}
\begin{proof} 
For $\leq$, replace $\alpha$ by $(\alpha^{<\kappa})^{<\kappa}$ in
Theorem~\ref{T1}(a) and use Theorem~\ref{expon}(d).

For $\geq$, it is enough to note from Corollary~\ref{C2.8} that
$$w((X_I)_\kappa)\geq w((\two^I)_\kappa)\geq
w((\two^{(\alpha^{<\kappa})})_\kappa)=(\alpha^{<\kappa})^{<\kappa}.$$
\vskip-20pt
\end{proof}

Like the authors, the reader will have noted already at this stage a
distinction in kind between the pleasing, clear-cut result given in
Discussion~\ref{oldweight} concerning the weight of a product in the
usual product topology and the less satisfactory statement given
in Corollary~\ref{wXIkappa}; in this latter, the weight of spaces of the
form $(\two^I)_\kappa$ is determined by $|I|$, but unexpectedly such
products which differ in size may have the same weight.

\begin{corollary}\label{betaleqalpha}
Let $\alpha$, $\kappa$ and $\lambda$ be infinite cardinals such that
$\lambda\leq\kappa$, and let $\{X_i:i\in I\}$ be a set
of spaces such that
$|I|=\alpha$, and $\two\subseteq_h X_i$ and $w(X_i)\leq\alpha$ for each
$i\in I$. Then
$w((X_I)_\lambda)\leq w((X_I)_\kappa)$.
\end{corollary}
\begin{proof}
Necessarily we have $\lambda\leq\kappa\leq\alpha^+$, or
$\lambda\leq\alpha^+\leq\kappa$, or $\alpha^+\leq\lambda\leq\kappa$. In
those three cases,
Theorem~\ref{T2.9} gives respectively

$w((X_I)_\lambda)=\alpha^{<\lambda}\leq\alpha^{<\kappa}=w((X_I)_\kappa)$,

$w((X_I)_\lambda)=\alpha^{<\lambda}\leq\alpha^\alpha=2^\alpha=w((X_I)_\kappa)$,
and

$w((X_I)_\lambda)=2^\alpha=w((X_I)_\kappa)$. 
\end{proof}

\begin{remarks}\label{R2.9}
{\rm
(a) Surely Corollary \ref{betaleqalpha} is as expected.
Presumably a short, direct proof is available but
the authors' search for that was
unsuccessful. We note however that, as the
simple example in Remark~\ref{R1.2}(b)({\it ii}) shows,
a larger topology (for example, the discrete topology) on a given set
may have a strictly smaller weight than does a smaller Tychonoff topology.

(b) The authors find surprising both the extent of validity of the formula
given in Theorem~\ref{T2.9} and the simplicity of its proof. We had
anticipated finding an explicit formula for
$w((X_I)_\kappa)$ only under
special axioms and assumptions (perhaps GCH, for example), and we had
anticipated the necessity to consider, at the least, such cardinals as
$\cf(\kappa)$, $\cf(\alpha)$ and $\log(\alpha)$, as well as the least
cardinal $\gamma$ such that $\alpha^\gamma>\alpha$.
}
\end{remarks}

\section{On the density character of $\kappa$-box products}

In this section we continue to investigate spaces of the form
$(X_I)_\kappa$, focusing now on the
invariant $d$ rather than on $w$.
Our point of departure and motivation is the paradigmatic trilogy of
Theorems~\ref{HMP}, \ref{HMPconv} and \ref{dexact}, which for the usual
product topology give respectively upper bounds, lower bounds, and
conditions of equality for (certain) numbers of the form
$d((X_I)_\kappa)$.
To avoid unnecessary restrictions, we state these three familiar results
in considerable generality. Standard treatments often impose stronger
separation properties according to authors' conventions, but the
published proofs (of Theorems~\ref{HMPconv} and \ref{dexact} in
\cite[3.19 and 3.20]{comfneg74}, for example) suffice to establish
Theorems~\ref{HMP}--\ref{dexact} in the form we have chosen.
Theorem~\ref{HMP} is, of course,
the classic theorem of Hewitt, Marczewski and
Pondiczery~\cite{hewitt46a},
\cite{marc47}, \cite{pondi}, stated here in two useful equivalent forms;
and Theorem~\ref{HMPconv} is its converse.

Our $\kappa$-box analogues to Theorems~\ref{HMP} and \ref{HMPconv} are
given in \ref{maindlemma}--\ref{HMPgen} and
\ref{d(Xkappa)}--\ref{optimal|I|}, respectively. The quest for the
exact $\kappa$-box analogue of Theorem~\ref{dexact}---that is, the
search for a specific cardinal number $\delta$ depending on the
variables $|I|$, $d(X_i)$ ($i\in I$) and $\kappa$ so that
$d((X_I)_\kappa)=\delta$---is elusive, perhaps
unattainable. For example, answering a question from 
\cite{comfneg72}, \cite{comfneg74}, Cater, Erd\"os, and Galvin 
\cite{ceg} have shown that in some models for $\beta = \aleph_\omega$
the inequalities 
$$d((\two^{(\beta^+)})_{\omega^+})=\cc<\beta=\log(2^\beta)<d((\two^{(2^\beta)})_{\omega^+})$$
occur. Furthermore, it has been known for some time~\cite{ceg},
\cite{comfrobii} that consistently 
$d((\two^\beta)_{\omega^+})=(\log(\beta))^\omega$ for every
infinite cardinal $\beta$. The question whether that equality holds in 
(all models of) ZFC, raised in ~\cite{ceg}, was answered 
in the negative by Gitik and Shelah \cite{gitikshelah}; we discuss 
their models in \ref{gitik-shelah}(d)--(g).

The foregoing paragraph explains why we are able for $\kappa>\omega$ to
offer exact computations of the form 
$d((X_I)_\kappa)=\delta$, in parallel with
Theorem~\ref{dexact}, only for spaces $X_i$ ($i\in I$) and $\kappa$
subject to severe constraints. Our (few) contributions of this sort are
given in Corollary~\ref{cortomain}(a) and
Theorems~\ref{HMPkappa}---\ref{T3.17} below.

\begin{theorem}\label{HMP}
{\rm [Version 1]} Let $\alpha\geq\omega$, $X_I=\Pi_{i\in I}\,X_i$ with
$d(X_i)\leq\alpha$ for each $i\in I$ and with
$|I|\leq2^\alpha$. Then
$d(X_I)\leq\alpha$.

{\rm [Version 2]} Let $I$ be an infinite set and $\{X_i:i\in I\}$ a set
of spaces. Then
$d(X_I)\leq\max\{\sup\{d(X_i):i\in I\},\log|I|\}$.
\end{theorem}

\begin{theorem}\label{HMPconv}
Let $\alpha\geq\omega$ and let $X_I=\Pi_{i\in I}\,X_i$ with $S(X_i)\geq3$
for each $i\in I$. If $d(X_i)>\alpha$ for some $i\in
I$, or if $|I|>2^\alpha$, then $d(X_I)>\alpha$.
\end{theorem}

\begin{theorem}\label{dexact}
If $\{X_i:i\in I\}$ is a family of spaces such that
$S(X_i)\ge 3$ for each $i\in I$ and
$|I|\ge\omega$, then
$$d(X_I)=\max\{\sup\{d(X_i):i\in I\},\log|I|\}.$$
\end{theorem}

\centerline{{\sc Part A}. {\sc Upper Bounds for } $d((X_I)_\kappa)$.}

We say that a subset $A$ of a space $X$ is {\it strongly
discrete} (in $X$) if there is a family $\{U(a):a\in A\}$ of pairwise
disjoint open subsets of $X$ such that $a\in U(a)$ for each $a\in A$.
Simple examples show that a strongly discrete set need not be closed. It
is clear, however, that if $\kappa$ is fixed and every
$A\in[X]^{<\kappa}$ is strongly discrete, then also every
$A\in[X]^{<\kappa}$ is closed in $X$. That motivates the
following terminology.

\begin{definition}\label{strongdis}
{\rm
Let $\kappa\geq\omega$ and let $X$ be a space. Then $X$ is
{\it strongly} $\kappa$-{\it discrete} if every
$A\in[X]^{<\kappa}$ is strongly discrete.
}
\end{definition}

\begin{remarks}
{\rm
(a) The terminology in Definition~\ref{strongdis} is not 
in universal usage. Note that the separation requirement applies only to sets
$A\in[X]^{<\kappa}$, not to all $A\in[X]^{\leq\kappa}$. Note also that
since in a strongly $\kappa$-discrete space $X$ each set $A\in[X]^{<\kappa}$
is both closed and discrete,  the condition is strictly stronger than the
condition that each discrete set
$A\in[X]^{<\kappa}$ is strongly discrete.

(b) We note that a space which for some $\kappa\geq\omega$ is strongly
$\kappa$-discrete is a Hausdorff space.

(c) We give the following lemma in the generality it warrants, but in
fact we will use it only when each of the spaces $E_i$
is discrete.
}
\end{remarks}

\begin{lemma}\label{L3.3}
Let $\kappa\geq\omega$ and let $E=E_I=\Pi_{i\in I}\,E_i$
with each space
$E_i$ strongly $\kappa$-discrete. Then $(E_I)_\kappa$ is strongly
$\kappa$-discrete.
\end{lemma}
\begin{proof}
Given $A\in[E]^{<\kappa}$ there is $J\in[I]^{<\kappa}$ such that the
projection
$\pi_J:E\twoheadrightarrow\Pi_{i\in J}\,E_i$, when restricted to $A$,
is an injection.
(If $\kappa>|I|$ we may take $J=I\in[I]^{<\kappa}$.) Now for
$i\in I$ and $a\in A$ we choose a neighborhood $U_i(a)$ of $a_i$ in
$X_i$ such that

(a) $U_i(a_i)=U_i(b_i)$ if $a,b\in A$ and $a_i=b_i$, and

(b) $U_i(a_i)\cap U_i(b_i)=\emptyset$ if $a,b\in A$ and $a_i\neq b_i$.

\noindent [Such a family $\{U_i(a_i):a\in A\}$ exists in $X_i$ since
$\pi_i[A]\in[X_i]^{<\kappa}$.] Then the sets
$U(a):=(\Pi_{i\in J}\,U_i(a_i))\times\Pi_{i\in I\backslash J}\,E_i$
are open in $(E_I)_\kappa$ and are pairwise disjoint with $a\in
U(a)$ for each $a\in A$.
\end{proof}

The principal result of this section is given in
Theorem~ \ref{maindtheorem}.
The following lemma does most of the work.

\begin{lemma}\label{maindlemma}
Let $\alpha\ge 2$, $\beta\ge\omega$, and $\kappa\ge\omega$ be cardinals and 
let $E:=(D(\alpha))^{2^{\beta}}$.
Then $d(E_\kappa)\leq\alpha^{<\kappa}\cdot(\beta^{<\kappa})^{<\kappa}$.
\end{lemma}

\begin{proof}
We set $X=\two^{\beta}$.
Since $w(\two)=2\leq\beta$ there is (by
Corollary~\ref{C2.4}) a base $\sB$ for
$(\two^{\beta})_\kappa$ such
that $|\sB|=\beta^{<\kappa}$. We assume without loss
of generality that
the elements of $\sB$ are drawn from the canonical base for
$(\two^{\beta})_\kappa$ (see~\cite[1.1.15]{E1}). Let
$\CC:=\{\sC\subseteq\sP(\sB):\sC~\mbox{is cellular in}~(\two^{\beta})_\kappa~\mbox{and}~~|\sC|<\kappa\},$
and for each $\sC\in\CC$ and $f:\sC\rightarrow D(\alpha)$ define
$p(\sC,f)\in E=(D(\alpha))^{\two^{\beta}}$ by
$$(p(\sC,f))_x= \left\{
\begin{array}
{r@{~\mathrm{if}~}l}
f(x)&\mbox{there is~} C\in\sC \mbox{~such that~} x\in C\\
0&x\in\two^{\beta}\backslash\bigcup\sC
\end{array} \right\}.$$

We set $A:=\{p(\sC,f):\sC\in\CC,f:\sC\rightarrow D(\alpha)\}$.

Since $\kappa\leq\beta^{<\kappa}=|\sB|$ we have 
$|[\sB]^{<\kappa}|=(\beta^{<\kappa})^{<\kappa}$.
 
Then since
$|\CC|\leq|[\sB]^{<\kappa}|=(\beta^{<\kappa})^{<\kappa}$ 
and for
$\sC\in\CC$ we have
$|\alpha^\sC|=\alpha^{|\sC|}\leq\alpha^{<\kappa}$-many
functions $f:\sC\rightarrow D(\alpha)$,
it follows that 
$$|A|\leq\alpha^{<\kappa}\cdot(\beta^{<\kappa})^{<\kappa}.$$
It suffices then to show that $A$ is dense in
$E_\kappa=((D(\alpha))^{2^{\beta}})_\kappa$.

Let $U=\Pi_{x\in\two^{\beta}}\,U_x$
be a canonical open subset of $E_\kappa$.
Without loss of generality we take $|U_x|=1$
when $x\in R(U)\in[\two^{\beta}]^{<\kappa}$ (and necessarily
$U_x=D(\alpha)$ when
$x\in\two^{\beta}\backslash R(U)$). We define
$f:R(U)\rightarrow D(\alpha)$ so
that $U_x=\{f(x)\}$. Since $(\two^{\beta})_\kappa$ is
strongly $\kappa$-discrete
(by Lemma~\ref{L3.3}) and $R(U)\in[\two^{\beta}]^{<\kappa}$,
there is a family
$\sC=\{C(x):x\in R(U)\}\in\CC$ of pairwise disjoint open subsets
of $(\two^{\beta})_\kappa$ such that $x\in C(x)$
for each $x\in R(U)$.
Then for $x\in R(U)$ we have
$x\in C(x)\in\sC\in\CC$,
and $(p(\sC,f))_x=f(x)\in U_x$; it follows that
$p(\sC,f)\in A\cap U$, as required.
\end{proof}

The proof of Lemma~\ref{maindlemma} seemed so natural that for some time
we considered its statement to be optimal. However, a stronger statement
is available. This is the principal result of this section,
given now in two equivalent formulations. 

\begin{theorem}\label{maindtheorem}
Let $\alpha\ge 2$, $\beta\ge\omega$, and $\kappa\ge\omega$ be cardinals.

{\rm (a)} 
Let $E:=(D(\alpha))^{2^{\beta}}$.
Then $d(E_\kappa)\leq(\alpha\cdot\beta)^{<\kappa}$.

{\rm (b)} 
Let $E:=(D(\alpha))^\beta$.
Then $d(E_\kappa)\leq(\alpha\cdot\log(\beta))^{<\kappa}$.
\end{theorem}

\begin{proof}
\,[When (a) is known, (b) follows upon replacing $2^\beta$ in (a) by
$\beta$ and using the inequality $\beta\leq2^{\log(\beta)}$. To derive
(a) from (b), replace $\beta$ in (b) by $2^\beta$ and use that 
$\log(2^\beta)\le\beta$.]

To prove (a), we consider two cases.

\underline{Case 1}. $\kappa$ is regular.
Then it follows from Lemma \ref{maindlemma} and Theorem \ref{expon}(b)
that
$$d(E_\kappa)\leq\alpha^{<\kappa}\cdot(\beta^{<\kappa})^{<\kappa}=
\alpha^{<\kappa}\cdot\beta^{<\kappa}=(\alpha\cdot\beta)^{<\kappa}.$$

\underline{Case 2}. $\kappa$ is singular (hence, a limit cardinal). 

[Here we use a trick taken from ~\cite[p.~308]{ceg}.]
For $\lambda<\kappa$ there is, by
Case~1 applied to the regular cardinal $\lambda^+$, a dense set
$A(\lambda)\subseteq E_{\lambda^+}$ such that
$|A(\lambda)|\leq(\alpha\cdot\beta)^{<\lambda^+}=(\alpha\cdot\beta)^\lambda$.
The set $A:=\bigcup_{\lambda<\kappa}\,A(\lambda)$
is clearly dense in $E_\kappa$, so
$$d(E_\kappa)\leq|A|\leq\Sigma_{\lambda<\kappa}\,(\alpha\cdot\beta)^\lambda=(\alpha\cdot\beta)^{<\kappa}.$$
\vskip-18pt
\end{proof}

\begin{corollary}\label{cortomain}
Let $\alpha\geq2$ and $\kappa\geq\omega$ be cardinals, and let
$1\leq\lambda\leq(\alpha^{<\kappa})^{<\kappa}$ and
$1\le\mu\leq2^{((\alpha^{<\kappa})^{<\kappa})}$. Then

{\rm (a)}
$d(((D((\alpha^{<\kappa})^{<\kappa}))^{2^{((\alpha^{<\kappa})^{<\kappa})}})_\kappa)=(\alpha^{<\kappa})^{<\kappa}$; and
 
{\rm (b)} $d(((D(\lambda))^\mu)_\kappa)\leq(\alpha^{<\kappa})^{<\kappa}$.
\end{corollary}
\begin{proof}
Clearly (b) is immediate from (a). To prove (a), it is enough
to replace $\alpha$ and $\beta$ in Theorem~\ref{maindtheorem} by
$(\alpha^{<\kappa})^{<\kappa}$ and then to use
Theorem~\ref{expon}(d).
\end{proof}

\begin{discussion}\label{D3.10}
{\rm
A convenient method of proof of the Hewitt-Marczewski-Pondiczery theorem
(Theorem~\ref{HMP}), adopted by many
expositors, is to prove first that the tractable
space $E:=(D(\alpha))^{2^\alpha}$
has a dense subset $A$ with $|A|=\alpha$; since evidently
there is a continuous
function $f$ from $E$ onto a dense subset of $X$, the set $f[A]$ is
dense in $X$, with $|f[A]|\leq|A|=\alpha$. The identical argument
suffices to derive Corollary~\ref{HMPgen}
from Theorem~\ref{maindtheorem} and Corollary~\ref{cortomain}.
}
\end{discussion}

\begin{corollary}\label{HMPgen}
Let $\alpha\ge 2$, $\beta\geq\omega$ and $\kappa\ge\omega$ be cardinals and
let $\{X_i:i\in I\}$ be a set of spaces.
\begin{itemize}
\item[(a)] If $d(X_i)\leq\alpha$ for each $i\in I$ and
$|I|\leq2^\beta$, then
$d((X_I)_\kappa)\leq(\alpha\cdot\beta)^{<\kappa}$;
\item[(b)] if $d(X_i)\leq\alpha$ for each
$i\in I$ and $|I|\leq\beta$,
then
$d((X_I)_\kappa)\leq(\alpha\cdot\log(\beta))^{<\kappa}$; and
\item[(c)] if $d(X_i)\leq(\alpha^{<\kappa})^{<\kappa}$ for each $i\in I$
and $|I|\leq2^{((\alpha^{<\kappa})^{<\kappa})}$,
then $d((X_I)_\kappa)\leq(\alpha^{<\kappa})^{<\kappa}$.
\end{itemize}
\end{corollary}

\begin{remark}
{\rm
For cardinals $\alpha$ and $\kappa$ such that $\alpha^{<\kappa}=\alpha=\beta$,  
Corollary \ref{HMPgen}(a) is \cite[3.18]{comfneg74} and was also
mentioned in \cite[p.~76]{comfneg82}.
}
\end{remark}

We restate Corollary~\ref{HMPgen}(b) in the form most easily comparable
with Version~2 of Theorem~\ref{HMP}.

\begin{theorem}\label{moreHMP}
Let $\{X_i:i\in I\}$ be a set of spaces with each $d(X_i)=\alpha_i$,
and let $\alpha:=\sup_{i\in I}\,\alpha_i$. Then
$$d((X_I)_\kappa)\leq\max\{\alpha^{<\kappa},(\log(|I|))^{<\kappa}\}.$$
\end{theorem}

As we see in Discussion~\ref{gitik-shelah}(d), however, the
inequality in Theorem~\ref{moreHMP}
can be strict. Thus consistently the obvious $\kappa$-box analogue of
Theorem~\ref{dexact} can fail.

\begin{discussion}\label{gitik-shelah}
{\rm
The two results
$$d(((D(\alpha))^{2^\alpha})_\kappa)\leq\alpha^{<\kappa}$$
and
$$d(((D((\alpha^{<\kappa})^{<\kappa}))^{2^{((\alpha^{<\kappa})^{<\kappa})}})_\kappa)=(\alpha^{<\kappa})^{<\kappa},$$
valid for $\alpha\geq2$ and $\kappa\geq\omega$ and
given by Theorem~\ref{maindtheorem}(a) and Corollary~\ref{cortomain}(a)
respectively,
suggest the attractive ``intermediate" speculation
\begin{equation}\label{Eq3}
d(((D(\alpha^{<\kappa}))^{2^{(\alpha^{<\kappa})}})_\kappa)\leq\alpha^{<\kappa}
\end{equation}
which, if valid, would yield these two weaker statements:
\begin{equation}\label{Eq4}
d(((D(\alpha))^{2^{(\alpha^{<\kappa})}})_\kappa)\leq\alpha^{<\kappa}
\end{equation}
and
\begin{equation}\label{Eq5}
d(((D(\alpha^{<\kappa}))^{2^\alpha})_\kappa)\leq\alpha^{<\kappa}.
\end{equation}

We discuss what we do and do not know about the truth value of 
(\ref{Eq3}), (\ref{Eq4}), and (\ref{Eq5}).

(a) (\ref{Eq3}) (\emph{hence also} (\ref{Eq4}) \emph{and} (\ref{Eq5})) 
\emph{holds for all $\alpha$ and $\kappa$ satisfying 
$(\alpha^{<\kappa})^{<\kappa}=\alpha^{<\kappa}$}. 
This is obvious from Corollary~\ref{cortomain}(a).
Thus by Theorem~\ref{weak<kappa}(a) the conditions (\ref{Eq3}), (\ref{Eq4}) and 
(\ref{Eq5}) hold (for all $\alpha\geq2$) when $\kappa$ is regular or there is
$\nu<\kappa$ such that $\alpha^\nu=\alpha^{<\kappa}$.

(b) (\ref{Eq5}) \emph{holds in} ZFC, \emph{for all} $\kappa\geq\omega$, 
\emph{when} $2\leq\alpha<\omega$. This is obvious, since
$|D(\alpha^{<\kappa})^{2^\alpha}|=\alpha^{<\kappa}$ in that case.

(c) \emph{for all $\alpha\geq\kappa$,} (\ref{Eq5}) \emph{fails} (\emph{hence} 
(\ref{Eq3}) \emph{fails}) \emph{in} ZFC \emph{for certain $\kappa$.}
In fact, we prove this statement:

{\it Let $\alpha\geq\omega$.
There are arbitrarily large cardinals $\kappa$ such that the space
$E:=(D(\alpha^{<\kappa}))^\alpha$
satisfies $d(E_\kappa)>\alpha^{<\kappa}$.}

To prove that, choose $\lambda\geq\omega$ such that
$\cf(\lambda)\leq\alpha$
(for example, set $\lambda:=\omega$). Then,
set $\kappa:=\beth_\lambda(\alpha)$. Since $\kappa>\alpha$
the space $E_\kappa$ is discrete, so from Remarks~\ref{bethprops}(b) we have
$$d(E_\kappa)=|E|=\kappa^\alpha\geq\kappa^{\cf(\lambda)}
=\kappa^{\cf(\kappa)}>\kappa=\alpha^{<\kappa}.$$

(d) \emph{Consistently,} (\ref{Eq4}) \emph{fails} (\emph{hence} (\ref{Eq3}) 
\emph{fails}) \emph{when $\alpha=2$ and $\kappa=\aleph_1$.}
Indeed, Gitik and Shelah~\cite{gitikshelah}, answering a question
left unresolved in \cite{ceg} and \cite{comfrobii}, have constructed models 
$\VV_1$ and $\VV_2$ of {\rm ZFC} such that
$$d((\two^{\aleph_\omega})_{\aleph_1})= \left\{
\begin{array}
{r@{\quad \mathrm{in}\quad}l}
\aleph_{\omega+1}            &  \VV_1  \\
\aleph_{\omega+2}   & \VV_2
\end{array} \right. ,$$
with $2^{\aleph_\omega}=\aleph_\omega^\omega=\aleph_{\omega+2}$  
in each case and with ``GCH below $\aleph_\omega$", so that
$2^{<\aleph_\omega}=\aleph_\omega$. Then, taking $\alpha=2$ and
$\kappa=\aleph_\omega$ in Theorem~\ref{maindtheorem}(a) we have
$$2^{\aleph_\omega}\ge d((\two^{(2^{(2^{<\aleph_\omega})})})_{\aleph_\omega})\ge
d((\two^{\aleph_\omega})_{\aleph_1})=\aleph_{\omega+2}=2^{\aleph_\omega}>\aleph_\omega=2^{<\aleph_\omega}$$
in the Gitik-Shelah model 
$\VV_2$, while in $\VV_1$ we have
$$2^{\aleph_\omega}\ge d((\two^{(2^{(2^{<\aleph_\omega})})})_{\aleph_\omega})\ge
d((\two^{\aleph_\omega})_{\aleph_1})=\aleph_{\omega+1}>\aleph_\omega=2^{<\aleph_\omega}.$$

Thus in both $\VV_1$ and $\VV_2$ we have 
$$2^{\aleph_\omega}\ge
d((\two^{(2^{(2^{<\aleph_\omega})})})_{\aleph_\omega})>2^{<\aleph_\omega},$$
\noindent so (\ref{Eq4}) (hence (\ref{Eq3})) fails there.

(e) We interpret the cited results of Gitik and Shelah, where the
density character of so simple a space as
$(\two^{\aleph_\omega})_{\aleph_1}$ is not determined by the axioms of
ZFC (even when $2^{\aleph_n}=\aleph_{n+1}$ for all $n<\omega$ and
$2^{\aleph_\omega}=\aleph_{\omega+2}$) as indicating the difficulty,
perhaps even the futility, of finding a pleasing and definitive
$\kappa$-box analogue of Theorem~\ref{dexact}.

(f) It is clear from the relations
$$|\two^{\aleph_\omega}|\geq d((\two^{\aleph_\omega})_{\aleph_\omega})\geq
d((\two^{\aleph_\omega})_{\aleph_1})
=\aleph_{\omega+2}=|\two^{\aleph_\omega}|$$
in $\VV_2$ that
$d((\two^{\aleph_\omega})_{\aleph_\omega})=\aleph_{\omega+2}$ there. 
For the value of $d((\two^{\aleph_\omega})_{\aleph_\omega})$
in the model $\VV_1$ we have 
$$\aleph_{\omega+2}=|2^{\aleph_\omega}|\geq
d((\two^{\aleph_\omega})_{\aleph_\omega})\geq
d((\two^{\aleph_\omega})_{\aleph_1})
=\aleph_{\omega+1}$$
there, i.e., 
$$\aleph_{\omega+1}\leq
d((\two^{\aleph_\omega})_{\aleph_\omega})\leq\aleph_{\omega+2}.$$
As it was noted in \cite[p. 236]{gitikshelah} there exist models of 
{\rm ZFC} such that 
$$d((D(\alpha)^{\aleph_\omega})_{\kappa})=\aleph_{\omega+1}$$
for every $\alpha,\kappa<\aleph_{\omega}$ and therefore in such models 
$d((\two^{\aleph_\omega})_{\aleph_\omega})=\aleph_{\omega+1}$.

We do not know if there exist models of {\rm ZFC} such that 
$d((\two^{\aleph_\omega})_{\aleph_1})=\aleph_{\omega+1}$ and 
$d((\two^{\aleph_\omega})_{\aleph_\omega})=\aleph_{\omega+2}$.

(g) For an exact computation of the weight and Souslin number of the
spaces $(\two^{\aleph_\omega})_{\aleph_1}$ and
$(\two^{\aleph_\omega})_{\aleph_\omega}$ in the Gitik-Shelah models
$\VV_1$ and $\VV_2$, see Remark~{\ref{S(G-S)} below.

(h) While we do not pretend to follow every detail of the arguments
from \cite{gitikshelah}, nor to frame maximal generalizations, we note
that the consistent failure of (\ref{Eq3}) and (\ref{Eq4}) is not restricted 
to the case $\alpha=2<\omega$. In both $\VV_1$ and $\VV_2$ one evidently has
$\aleph_n^{<\aleph_\omega}=\aleph_\omega$ for $0<n<\omega$, so (\ref{Eq3})
and (\ref{Eq4}) fail in those models (with $\kappa=\aleph_\omega$) for every 
$\alpha$ such that $2\leq\alpha<\aleph_\omega$.

(i) In passing we note the existence of two misprints in \cite{gitikshelah}
which have confused at least two
readers: Reference in Theorem~1.1(c) should be
only to {\it uncountable} cardinals $\gamma$, and in Theorem~4.2(4) the symbol
$<\aleph_0$ should be $<\aleph_1$.
}}
\end{discussion}

\begin{remark}\label{R3.18}
{\rm
The arguments developed to prove \ref{maindlemma}--\ref{moreHMP}
follow the general pattern of
classical arguments used to prove the original
Hewitt-Mar\-czewski-Pondiczery Theorem~\ref{HMP}, albeit with
combinatorial modifications necessary to accommodate to the $\kappa$-box
topology. (When $\kappa=\omega$, Lemma~\ref{L3.3} reduces to the simple
observations that
(1)~the product of Hausdorff spaces is a Hausdorff space, and
(2)~in a Hausdorff space, the points of any finite set can be separated
by disjoint open sets.)
Quite likely, it was reasoning similar to ours which over 40 years ago
provoked from Engelking and Kar{\l}owicz \cite[p.~285]{EnKa}, after they
had completed their own proof of the Hewitt-Mar\-czewski-Pondiczery
theorem, the cavalier sta\-te\-ment (here we quote faithfully, but using
the notation of the present paper) ``We can also derive theorems analogous to
those above for $\kappa$-box topologies$\,\ldots$. We shall not formulate
these theorems since they are less interesting, but the reader, if he
wishes, will be able to do so without the least difficulty." OK, fair
enough. We do note, however, that in the several treatments known to us of
the Hewitt-Marczewski-Pondiczery theorem, we have
found no mention of the cardinal number $(\alpha^{<\kappa})^{<\kappa}$
which figures prominently and naturally in our development.
(This is hardly surprising
with respect to the paper \cite{EnKa}, since those authors
restrict attention to
box products of the form $(X_I)_{\kappa^+}$.) Nor have we found an
indication, as in Theorem~\ref{HMPkappa} below, that the upper bound
$\alpha^{<\kappa}\geq d(((D(\alpha))^{\beta})_\kappa)$
given in
Theorem~\ref{maindtheorem}(a) is in fact assumed in every case with
$\kappa\leq\beta\leq2^\alpha$.
}
\end{remark}

\centerline{{\sc Part B}. {\sc Lower Bounds for $d((X_I)_\kappa)$}}

In Part A, seeking $\kappa$-box analogues and generalizations of
Theorem~\ref{HMP}, for specific function pairs $f$ and $g$ of two
variables we have sought a function $h$ so that
$$d(((D(f(\alpha,\kappa)))^{g(\alpha,\kappa)})_\kappa)\leq h(\alpha,\kappa)$$
holds. Now in Part~B, again for hand-picked $f$ and $g$, we seek $h'$ so
that
\begin{equation}\label{Eq6}
d(((D(f(\alpha,\kappa)))^{g(\alpha,\kappa)})_\kappa)\geq
h'(\alpha,\kappa).
\end{equation}

In some cases the choice $h=h'$ is accessible, so
$$d(((D(f(\alpha,\kappa)))^{g(\alpha,\kappa)})_\kappa)$$
is computed exactly. In other cases,
in parallel with Theorem~\ref{HMPconv}, we find several
conditions sufficient to ensure that the inequality (\ref{Eq6}) is strict.

\begin{lemma}\label{lowerbound}
Let $\alpha\geq2$ and $\kappa\ge\omega$ be cardinals and let
$E:=(D(\alpha))^\kappa$. Then 
$d(E_\kappa)\ge\alpha^{<\kappa}$.
\end{lemma}

\begin{proof}
If the inequality fails there is $\lambda<\kappa$ such that
$d(E_\kappa)<\alpha^\lambda$. The space
$((D(\alpha))^\lambda)_\kappa$ is then discrete, and since the
projection from $E_\kappa$ onto
$((D(\alpha))^{\lambda})_{\kappa}$ is continuous we have the contradiction
$$\alpha^\lambda=d(((D(\alpha))^\lambda)_\kappa)\leq d(E_\kappa)<\alpha^\lambda.$$
\vskip-18pt
\end{proof}

As we noted in Discussion~\ref{D3.10},
the Hewitt-Marczewski-Pondiczery theorem may
be regarded as a routine generalization of this startling special case:
$d((D(\alpha))^\beta)=\alpha$ when $\alpha\geq\omega$ and 
$1\le\beta\le 2^\alpha$. We draw specific attention therefore to the 
correct $\kappa$-box analogue of that result. We note that no 
regularity hypothesis is imposed here on the cardinal number $\kappa$.

\begin{theorem}\label{HMPkappa}
Let 
$\omega\leq\kappa\leq \beta\le 2^\alpha$. Then 
$d((D(\alpha))^\beta)_\kappa)=\alpha^{<\kappa}$.
\end{theorem}
\begin{proof}
The inequalities $\geq$ and $\leq$ are immediate from
Lemma \ref{lowerbound} and Theorem \ref{maindtheorem}(a), respectively.
\end{proof}

In the following theorems we compute the density character of
certain specific spaces.

\begin{theorem}\label{regularSLC}
Let $\omega\leq\kappa$ and $2\leq\alpha\leq\kappa$.
If either $\log(\kappa)<\kappa$ or $\kappa$ is
a regular strong limit cardinal, then

{\rm (a)} $2^{<\kappa}=\alpha^{<\kappa}=\kappa^{<\kappa}$; and

{\rm (b)} $d(((D(\alpha))^\kappa)_\kappa)=
2^{<\kappa}=\alpha^{<\kappa}=\kappa^{<\kappa}$.
\end{theorem}
\begin{proof} (a) That
$2^{<\kappa}\leq\alpha^{<\kappa}\leq\kappa^{<\kappa}$ is clear,
since $2\leq\alpha\leq\kappa$. Now if $\log(\kappa)<\kappa$ then
$$\kappa^{<\kappa}\leq(2^{\log(\kappa)})^{<\kappa}
=\Sigma_{\lambda<\kappa}\,(2^{\log(\kappa)})^\lambda=\Sigma_{\lambda<\kappa}\,2^\lambda=2^{<\kappa};$$
and if $\kappa$ is regular then since no set in $[\kappa]^{<\kappa}$ is
cofinal in $\kappa$ we have
$[\kappa]^{<\kappa}\subseteq\bigcup_{\eta<\kappa}\,\sP(\eta)$, so if in
addition $\kappa$ is a strong limit cardinal then
$$\kappa^{<\kappa}=|[\kappa]^{<\kappa}|\leq\Sigma_{\eta<\kappa}\,|\sP(\eta)|
=\Sigma_{\eta<\kappa}\,2^{|\eta|}\leq\Sigma_{\eta<\kappa}\,\kappa
=\kappa\leq2^{<\kappa}.$$

(b) From Theorem~\ref{HMPkappa} (with $\alpha=\beta=\kappa$ there) and
Lemma~\ref{lowerbound} (with $\alpha=2$ there) we have
$$\kappa^{<\kappa}=d(((D(\kappa))^\kappa)_\kappa)\geq
d(((D(\alpha))^\kappa)_\kappa)\geq d((\two^\kappa)_\kappa)\geq2^{<\kappa},$$
so the asserted equations follow from (a).
\end{proof}

Here is our most comprehensive result for numbers of the form 
$d((X_I)_\kappa)$.

\begin{theorem}\label{upper-lower}
Let $\alpha\geq2$ and $\kappa\geq\omega$ be cardinals, and let
$\alpha\leq\lambda\leq(\alpha^{<\kappa})^{<\kappa}$ and
$\kappa\le\mu\leq2^{((\alpha^{<\kappa})^{<\kappa})}$.
Then

{\rm (a)} $\alpha^{<\kappa}
\leq d(((D(\lambda))^\mu)_\kappa)\leq(\alpha^{<\kappa})^{<\kappa}$;

{\rm (b)} if $\kappa$ is regular
or some $\nu<\kappa$ satisfies $\alpha^\nu=\alpha^{<\kappa}$, 
then $d(((D(\lambda))^\mu)_\kappa)=\alpha^{<\kappa}$.
\end{theorem}
\begin{proof} (a) This is clear from Lemma~\ref{lowerbound}
and Corollary~\ref{cortomain}(b).

(b) From Theorem~\ref{weak<kappa}(a) we have
$\alpha^{<\kappa}=(\alpha^{<\kappa})^{<\kappa}$, so (b) follows from (a).
\end{proof}

We note next that for
$\lambda=\alpha$ and $\kappa\leq\mu\leq2^{(\alpha^{<\kappa})}$,
the conclusion of Theorem~\ref{upper-lower}(b) can be established with a
supplementary hypothesis weaker than the existence of
$\nu<\kappa$ such that $\alpha^\nu=\alpha^{<\kappa}$. (The
ZFC-consistent existence of instances to which Theorem~\ref{T3.17}
applies, while Theorem~\ref{upper-lower}(b) does not, is shown in
Remark~\ref{R3.20}.)

\begin{theorem}\label{T3.17}
Let $\alpha\geq2$, $\kappa\geq\omega$ and
$\kappa\leq\mu\leq2^{(\alpha^{<\kappa})}$, and set
$E:=(D(\alpha))^\mu$. If there is $\nu<\kappa$ such
that $2^{(\alpha^\nu)}=2^{(\alpha^{<\kappa})}$, then
$d(E_\kappa)=\alpha^{<\kappa}$.
\end{theorem}
\begin{proof}
That $d(E_\kappa)\geq\alpha^{<\kappa}$ is immediate from
Lemma~\ref{lowerbound}. Now
for $\nu\leq\lambda<\kappa$ we have
$2^{(\alpha^\lambda)}=2^{(\alpha^{<\kappa})}$ and
there is, 
by Theorem~\ref{maindtheorem}(a)
with $\alpha$, $\alpha^\lambda$, and $\lambda^+$ in the role of
$\alpha$, $\beta$, and $\kappa$ there, a dense 
set $A(\lambda)\subseteq
((D(\alpha))^{2^{(\alpha^{\lambda})}})_{\lambda^+}$
such that
$|A(\lambda)|\leq(\alpha\cdot\alpha^\lambda)^\lambda=\alpha^{\lambda}$.
The set
$A:=\bigcup_{\nu\le\lambda<\kappa}\,A(\lambda)$
is clearly dense in
$(((D(\alpha))^{(2^{(\alpha^{<\kappa})})})_\kappa$, so
$d(E_\kappa)\leq
d((((D(\alpha))^{(2^{(\alpha^{<\kappa})})})_\kappa)\leq
|A|\leq\Sigma_{\nu\le\lambda<\kappa}\,\alpha^\lambda\le
\kappa\cdot\alpha^{<\kappa}=\alpha^{<\kappa}.$
\end{proof}

\begin{remarks}\label{R3.20}
{\rm
(a) We indicate that there are models $\MM$ of ZFC in which, for suitably chosen
$\alpha$ and $\kappa$ as in Theorem~\ref{T3.17} (specifically for
$\alpha=2$, $\kappa=\aleph_\omega$) there exist $\nu<\kappa$ such that
$2^{\alpha^\nu}=2^{(\alpha^{<\kappa})}$ but there is no $\nu<\kappa$
such that $\alpha^\nu=\alpha^{<\kappa}$. To that end, using the
fundamental consistency theorem of Easton~\cite{easton70} as exposed by
Kunen~\cite[VIII]{kunen}, let $\MM$ be a model of ZFC in which
\begin{itemize}
\item[(1)]$2^{\aleph_n}=\aleph_{\omega+n+1}$ for $n<\omega$,
\item[(2)]$2^{\aleph_\omega}=\aleph_{\omega+\omega+1}$, and
\item[(3)]$2^{\aleph_{\omega+n+1}}=2^{\aleph_{\omega+\omega}}=\aleph_{\omega+\omega+2}$
 for $n<\omega$. 
\end{itemize}

It is clear in $\MM$, taking $\alpha=2$ and $\kappa=\aleph_\omega$, that
$$\alpha^{<\kappa}=2^{<\aleph_\omega}=\aleph_{\omega+\omega},$$
so for every $\nu=\aleph_n<\aleph_\omega=\kappa$ we have
$$\alpha^\nu=2^{\aleph_n}=\aleph_{\omega+n+1}<
\aleph_{\omega+\omega}=\alpha^{<\kappa}$$ and
$$2^{\alpha^\nu}=2^{(2^{\aleph_n})}=\aleph_{\omega+\omega+2}
=2^{\aleph_{\omega+\omega}}=2^{(\alpha^{<\kappa})}.$$

(b) We note in passing that the existence of $\nu<\kappa$ such that
$2^{\alpha^\nu}=2^{(\alpha^{<\kappa})}$ holds in all models in which
$\log(2^{(\alpha^{<\kappa})})<\alpha^{<\kappa}$.
Indeed if 
$\log(2^{(\alpha^{<\kappa})})=\beta<\alpha^{<\kappa}$ then there is
$\lambda<\kappa$ such that $\beta<\alpha^\lambda$, and then
$2^\beta=2^{\alpha^\nu}=2^{(\alpha^{<\kappa})}$ for all $\nu$ satisfying
$\lambda\leq\nu<\kappa$.
}
\end{remarks}

Next as promised we give a couple of generalizations of
Theorem~\ref{HMPconv} to the $\kappa$-box context.

\begin{theorem}\label{d(Xkappa)}
Let $\alpha\ge \omega$, $\beta\geq2$, $\kappa\ge \omega$ and
let $S(X_i)>\beta$ for each $i\in I$.
Suppose that either

{\rm (i)} some $i\in I$ satisfies $d(X_i)>\alpha$; or

{\rm (ii)} $|[I]^{<\kappa}|>2^\alpha$; or

{\rm (iii)} $\beta^{<\kappa}>2^\alpha$; or

{\rm (iv)} there is $J\in[I]^{<\kappa}$ such that $\beta^{|J|}>\alpha$.

\noindent Then
$d((X_I)_\kappa)>\alpha$.
\end{theorem}
\begin{proof}
The sufficiency of (i) is clear: the natural projection from
$(X_I)_\kappa$ to $X_i$ is continuous and surjective, so
$d((X_I)_\kappa)\geq d(X_i)$.

For the rest of the proof, for $i\in I$ let
$\{U_i(\eta):\eta<\beta\}$ be a cellular family in $X_i$.

We prove $d((X_I)_\kappa)>\alpha$, assuming that either (ii) or
(iii) holds.
For $A\in[I]^{<\kappa}$ and $f\in \beta^A$ set
$$U(A,f):=\{x\in X_I:i\in A\Rightarrow x_i\in U_i(f(i))\}.$$
Let $T$ be dense in $(X_I)_\kappa$ with $|T|=d((X_I)_\kappa)$ and
set $T(A,f):=T\cap U(A,f)$.

We claim that the map
$\phi:\bigcup_{A\in[I]^{<\kappa}}\,(A\times\beta^A)\longrightarrow\sP(T)$ 
given by $\phi(A,f)=T(A,f)$ is injective. Let $(A,f)\neq(B,g)$
with $A,B\in[I]^{<\kappa}$, $f\in\beta^A$, and $g\in\beta^B$. 
We consider two cases.

\underline{Case 1}. $A=B$. Then there is $i\in A=B$ such that $f(i)\neq g(i)$, so
$T(A,f)\cap T(B,g)=\emptyset$ (since $U_i(f(i))\cap U_i(g(i))=\emptyset$).

\underline{Case 2}. $A\neq B$. Without loss of generality there is then
$i\in A\backslash B$. Choose $\eta<\beta$ such that $f(i)\neq\eta$. The set
$V:=U(B,g)\cap\pi_i^{-1}(U_i(\eta))$ is then nonempty and open in
$(X_I)_\kappa$, and with $p\in T\cap V$ we have $p\in T(B,g)\backslash
T(A,f)$. 

The claim is proved.

For $\lambda<\kappa$ and $A\in[I]^\lambda$ we have
$A\times\beta^A\subseteq \dom(\phi)$, so $|\dom(\phi)|\geq|[I]^\lambda|$ and
$|\dom(\phi)|\geq\beta^\lambda$. Then it follows that
$|\dom(\phi)|\geq|[I]^{<\kappa}|\cdot\beta^{<\kappa}$, so if
$d((X_I)_\kappa)=|T|\leq\alpha$ and (ii) or (iii) holds
we would have the contradiction
$$2^\alpha<|[I]^{<\kappa}|\cdot\beta^{<\kappa}\leq|\dom(\phi)|
\leq|\sP(T)|\leq2^\alpha.$$

It remains to derive $d((X_I)_\kappa)>\alpha$ from (iv).
Let $J$ be as hypothesized,
and for $f\in\beta^J$ set
$$V(f):=(\Pi_{i\in J}\,U_i(f(i)))\times(\Pi_{i\in I\backslash J}\,X_i).$$
Then $\sV:=\{V(f):f\in\beta^J\}$ is cellular in
$(X_I)_\kappa$, so
$$d((X_I)_\kappa)\geq|\sV|=|\beta^J|=\beta^{|J|}>\beta.$$
\vskip-18pt
\end{proof}

We note that the hypothesis in Theorem~\ref{d(Xkappa)} on the family
$\{X_i:I\in I\}$ can be relaxed in places. In connection with (iv), for
example, it is clear that the condition $S(X_i)>\beta$ need hold only
for $i$ in some set $J\in[I]^{<\kappa}$ such that $\beta^{|J|}>\alpha$.

Taking $\beta=2$ in Theorem~\ref{d(Xkappa)} and replacing $\alpha$ there
first by $\alpha^{<\kappa}$ and then by $(\alpha^{<\kappa})^{<\kappa}$,
we obtain respectively parts (a) and (b) of the following corollary.

\begin{corollary}\label{optimal|I|}
Let $\alpha\ge 2$ and $\kappa\ge\omega$ be cardinals, and let
$\{X_i:i\in I\}$ be a set of spaces such that $S(X_i)\geq3$ for each
$i\in I$.

{\rm(a)} If $|[I]^{<\kappa}|>2^{(\alpha^{<\kappa})}$ then
$d((X_I)_\kappa))>\alpha^{<\kappa}$; and

{\rm(b)} if $|[I]^{<\kappa}|> 2^{((\alpha^{<\kappa})^{<\kappa})}$
then 
$d((X_I)_\kappa)>(\alpha^{<\kappa})^{<\kappa}$.
\end{corollary}

Corollary~\ref{optimal|I|} shows that the inequalities given in
Corollary~\ref{HMPgen} are sharp. 
The following simple combinatorial result offers reformulations of
some of the
hypotheses of Corollary~\ref{optimal|I|}.

\begin{theorem}\label{L3.25}
Let $\alpha\ge 2$ and $\kappa\ge\omega$ be 
cardinals and let $I$ be a set.

{\rm (a)} These three conditions are equivalent:

{\rm (1)} $|I|>2^{(\alpha^{<\kappa})}$;
{\rm (2)} $|[I]^{<\kappa}|>2^{(\alpha^{<\kappa})}$;
{\rm (3)} $|I|^{<\kappa}>2^{(\alpha^{<\kappa})}$.

{\rm(b)} These three conditions are equivalent.

{\rm (1)} $|I|>2^{((\alpha^{<\kappa})^{<\kappa})}$;
{\rm (2)} $|[I]^{<\kappa}|>2^{((\alpha^{<\kappa})^{<\kappa})}$;
{\rm (3)} $|I|^{<\kappa}>2^{((\alpha^{<\kappa})^{<\kappa})}$.
\end{theorem}
\begin{proof}
The implications (1)~$\Rightarrow$~(2) and (2)~$\Rightarrow$~(3) are
clear in both (a) and (b). To see that (3)~$\Rightarrow$~(1) in (a),
note that if $|I|\leq2^{(\alpha^{<\kappa})}$
then
$$|I|^{<\kappa}\leq(2^{(\alpha^{<\kappa})})^{<\kappa}
\leq(2^{(\alpha^{<\kappa})})^{\kappa}=
2^{(\alpha^{<\kappa})\cdot\kappa}=2^{(\alpha^{<\kappa})}.$$

The proof that (3)~$\Rightarrow~$(1) in (b) is similar.
\end{proof}

\begin{remarks}
{\rm
(a) The authors of \cite[3.16]{comfneg74}, improving
their results from~\cite{comfneg72}, show that if $E=(D(\alpha))^{2^\alpha}$ or
$E=(D(\alpha))^{\alpha^+}$, then $d(E_\kappa)=\alpha$ if and only if
$\alpha=\alpha^{<\kappa}$. Clearly Theorem~\ref{upper-lower}(b) above
improves that statement. Similarly, Corollary~\ref{HMPgen}(a) improves
\cite[3.18]{comfneg74}, which asserts the conclusion of \ref{HMPgen} only
under the assumption that $\alpha=\alpha^{<\kappa}$.

(b) The investigation by Hu~\cite{hu06} of cardinals of the form
$d((X_I)_\kappa)$ is from a different perspective:
Rather than beginning with the set $E=\Pi_{i\in I}\,D(\alpha_i)$ and
seeking dense subsets of the space $E_\kappa$,
Hu~\cite{hu06} uses (maximal)
generalized independent families of partitions
of a given set $S$ to map
$S$ faithfully onto dense subsets of spaces of the form
$E_\kappa$ (one writes $S\subseteq E_\kappa$). The emphasis
is on finding conditions so that $S\subseteq E_\kappa$ is irresolvable.
Hu~\cite{hu06} shows, for example, that if each
$\alpha_i$ is less than the first cardinal which is strongly
$\kappa$-inaccessible,
and $E_\kappa$ contains a dense, irresolvable subspace,
then $\kappa=2^{<\kappa}$, and consistently a measurable
cardinal exists.
}
\end{remarks}

\section{On the Souslin number of $\kappa$-box products}

We remind the reader of our standing convention that hypothesized spaces
are not assumed to enjoy any special separation properties. This
complicates our exposition slightly, since it is convenient for us to
cite some basic familiar results from sources where, for simplicity and
often unnecessarily, such properties as Hausdorff separation are assumed
throughout. We mention in particular the following two useful results,
both valid for every space. These will be used frequently in what
follows, without explicit restatement.

{\it Let $X$ be a space. Then

{\rm (a)} $S(X)\neq\omega$; and

{\rm (b)} either $S(X)<\omega$ or $S(X)$ is an (infinite) regular
cardinal.}\\
\noindent The proof given in~\cite[2.10]{comfneg82} of (a), although
long-winded and unnecessarily complicated, is valid without separation
assumptions; (b) is a fundamental result of
Erd{\H{o}}s and Tarski~\cite{erdtarski43} (see
\cite[2.10]{comfneg74}, \cite[2.14]{comfneg82} for other treatments).

As with Sections 2 and 3 concerning weight and density character
respectively, we begin this section by citing those classical Souslin-related
theorems (pertaining to the usual product topology) whose $\kappa$-box
analogues we study here. As usual, when a set
$\{X_i:i\in I\}$ of spaces is given we write $X_I:=\Pi_{i\in I}\,X_i$.

\begin{theorem}\label{firstSP}
Let $\{X_i:i\in I\}$ be a set of nonempty spaces and set
$$\alpha:=\sup\{S(X_F):\emptyset\neq F\in[I]^{<\omega}\}.$$
Then
$$S(X_I)= \left\{
\begin{array}
{r@{~\mathrm{~}~}l}
\alpha&\mbox{\rm{if~(a)~$\alpha<\omega$ or (b)~$\alpha$ is regular and }~}\alpha>\omega\\
\alpha^+&\mbox{\rm{in all other cases }~}
\end{array} \right\}.$$
\end{theorem}

The thrust of Theorem~\ref{firstSP} is that the Souslin number of a
product space $X_I$ (in the usual product topology) is completely
determined by the Souslin numbers of the various subproducts $X_F$ with
$F\in[I]^{<\omega}$. Much of this Section is devoted to the presentation
of $\kappa$-box
analogues of Theorem~\ref{firstSP}. See in particular 
Theorems~\ref{CN}, \ref{GCN}, \ref{T4.15}, \ref{GCH}, and
Corollaries \ref{S(power)}(b), \ref{CGCH}.

The proof of
Theorem~\ref{firstSP} depends on
nontrivial combinatorial machinery in which, reflecting the restriction
to the usual product topology, the cardinal numbers $\omega$ and
$\omega^+$ figure prominently.
The key to the proof is
the theory of quasi-disjoint families as developed by Erd{\H{o}}s and
Rado~\cite{erdosrado60}, \cite{erdosrado69} (the ``$\Delta$-system
lemma"); this is used in the proof of Theorem~\ref{CN}. For a thorough
development of that result and of several other Souslin-related
consequences, the reader may consult \cite[3.8]{comfneg74} and
\cite[3.25]{comfneg82}.

As we noted in Theorem~\ref{HMP}, for $\alpha\geq\omega$ the product of
$2^\alpha$-many (or fewer) spaces $X_i$, with each $d(X_i)\leq\alpha$,
satisfies $d(X_I)\leq\alpha$. From that and Theorem~\ref{firstSP}, one
can derive this 
well-known theorem (see for example \cite[2.3.17]{E1};
or see Theorem \ref{T4.12} for the general $\kappa$-box statement of which
Theorem~\ref{secondSP} is the case $\kappa=\omega$).

\begin{theorem}\label{secondSP}
Let $\alpha\geq\omega$ and $\{X_i:i\in I\}$
be a set of spaces with $d(X_i)\leq\alpha$ for each $i\in I$.
Then $S(X_I)\leq\alpha^+$.
\end{theorem}

\begin{discussion}
{\rm
Theorems \ref{firstSP} and \ref{secondSP} leave unanswered even
for the usual product topology a
question which arises naturally in their wake:

{\it Given $\alpha\geq\omega$ and a finite set $\{X_i:i\in F\}$ of
spaces with each $S(X_i)\leq\alpha$, is necessarily $S(X_F)\leq\alpha$?}

The brief response is that the question is not settled by the
axioms of ZFC, even in the case $\alpha=\omega^+$.
Referring the reader to \cite{comfneg82} for extensive comments and
relevant bibliographic citations, we remark simply
that it has been known in ZFC for many
years that while the Souslin number may ``jump" in passing from a space
$X$ to $X\times X$, roughly speaking that jump is bounded by a
single exponential. To be more precise we give below a theorem taken from 
\cite[3.13]{comfneg74} that gives one possible generalization of that claim for 
the $\kappa$-box topology. For its full generalization to the $\kappa$-box 
context, see Theorem \ref{GKurepa}.
}
\end{discussion}

\begin{theorem}
If $\alpha\ge\omega$ and $\{X_i:i\in I\}$ is a family of spaces such that 
$S(X_i)\le\alpha^+$ for $i\in I$, then $S((X_I)_{\alpha^+})\le(2^\alpha)^+$.
\end{theorem}

The following notational
device (see \cite[p.~254]{comfneg82}) is useful as we seek
$\kappa$-box analogues of
Theorems~\ref{firstSP} and \ref{secondSP}.

\begin{notation}
{\rm
Let $\alpha$ and $\kappa$ be infinite cardinals. Then $\alpha$ is {\it
strongly} $\kappa$-{\it inaccessi\-ble} (in symbols: $\kappa\ll\alpha$)
if (a)~$\kappa<\alpha$ and (b)~$\beta^\lambda<\alpha$ whenever
$\beta<\alpha$ and $\lambda<\kappa$.
}
\end{notation}

\begin{remark}\label{help}
{\rm
To help the reader fix ideas, we note that the condition
$\kappa\ll\alpha$ occurs for many pairs of cardinals. For example,

(1) every uncountable cardinal $\alpha$ satisfies $\omega\ll\alpha$;

(2) every infinite cardinal $\alpha$ satisfies $\alpha^+\ll(2^\alpha)^+$,
since if $\lambda<\alpha^+$ and $\beta<(2^\alpha)^+$ then
$\lambda\leq\alpha$ and $\beta\leq2^\alpha$ and hence
$\beta^\lambda\leq(2^\alpha)^\alpha=2^\alpha<(2^\alpha)^+$; and

(3) every pair $\kappa,\alpha$ with $\alpha\geq2$ and $\kappa\geq\omega$
satisfies
$\kappa\ll((\alpha^{<\kappa})^{<\kappa})^+$,
since if $\lambda<\kappa$ and $\beta<((\alpha^{<\kappa})^{<\kappa})^+$
then
$\beta\leq(\alpha^{<\kappa})^{<\kappa}$ and hence
$\beta^\lambda\leq((\alpha^{<\kappa})^{<\kappa})^{<\kappa}=(\alpha^{<\kappa})^{<\kappa}$
 (by Theorem~\ref{expon}(d)).

(4) every pair $\kappa,\alpha$ with $\alpha\geq2$ and 
$\kappa\geq\omega$ singular satisfies
$\kappa^+\ll((\alpha^{<\kappa})^{<\kappa})^+$,
since if $\lambda<\kappa^+$ and $\beta<((\alpha^{<\kappa})^{<\kappa})^+$
then $\lambda\le\kappa$ and $\beta\leq(\alpha^{<\kappa})^{<\kappa}$ and 
hence $$\beta^\lambda\leq((\alpha^{<\kappa})^{<\kappa})^{\kappa}=\alpha^{\kappa}=(\alpha^{<\kappa})^{<\kappa}$$ (by Theorem~\ref{expon}(c)).
}
\end{remark}

\begin{theorem}[{cf. \cite[3.8]{comfneg74},
\cite[3.25(a)]{comfneg82}}]\label{CN}
Let $\omega\le\kappa\ll\alpha$ with $\alpha$ regular
and let $\{X_i:i\in I\}$ be a family of nonempty spaces. 
Then $S((X_I)_\kappa)\le\alpha$ if and only if $S((X_J)_\kappa)\le\alpha$
for each nonempty $J\in [I]^{<\kappa}$.
\end{theorem}

From the relations $\alpha^+\ll(2^\alpha)^+$ and
$\kappa\ll((\alpha^{<\kappa})^{<\kappa})^+$ we have
these consequences of Theorem~\ref{CN}.

\begin{theorem}\label{GCN}
Let $\kappa\ge \omega$ and $\alpha\ge 2$ be cardinals and let 
$\{X_i:i\in I\}$ be a family of nonempty spaces. Then

{\rm (a)} If $\alpha\geq\omega$, then
$S((X_I)_{\alpha^+})\leq(2^\alpha)^+$ if and only if
$$S((X_J)_{\alpha^+})\leq(2^\alpha)^+$$ for each nonempty
$J\in[I]^{\leq\alpha}$; and
 
{\rm (b)} $S((X_I)_\kappa)\le((\alpha^{<\kappa})^{<\kappa})^+$ 
if and only if $$S((X_J)_\kappa)\le((\alpha^{<\kappa})^{<\kappa})^+$$
for each nonempty $J\in [I]^{<\kappa}$.
\end{theorem}

The following result, another immediate consequence of Theorem~\ref{CN},
furnishes in certain cases an exact
formula for the numbers $S((X_I)_\kappa)$.

\begin{corollary}\label{S=sup}
Let $\kappa\geq\omega$, let $\{X_i:i\in I\}$ be a set of nonempty
spaces, and set
$\alpha:=\sup\{S((X_J)_\kappa):\emptyset\neq J\in[I]^{<\kappa}\}.$
Then

{\rm (a)} {\rm (cf. \cite[3.27]{comfneg82})} If $\alpha$ is regular and 
$\kappa\ll\alpha$ then $S((X_I)_\kappa)=\alpha$;

{\rm (b)} if $\alpha$ is singular and $\kappa\ll\alpha^+$,
then $S((X_I)_\kappa)=\alpha^+$.
\end{corollary}

The following result, taken from \cite[3.28]{comfneg82}, is given here
for the reader's convenience. Since every infinite cardinal
$\alpha$ satisfies $\omega\ll\alpha^+$ with
$\alpha^+$ regular, the implication (a)~$\Rightarrow$~(b) is
a suitable $\kappa$-box analogue of
Theorem~\ref{secondSP}.

\begin{theorem}\label{iffk<<a}
Let $\omega\leq\kappa<\alpha$ with $\alpha$ regular. Then these
conditions are equivalent.

{\rm (a)} $\kappa\ll\alpha$;

{\rm (b)} if $\{X_i:i\in I\}$ is a set of spaces with each
$d(X_i)<\alpha$, then $S((X_I)_\kappa)\leq\alpha$.
\end{theorem}
\begin{proof} (a)~$\Rightarrow$~(b). According to Theorem~\ref{CN}, it
suffices to show that $S((X_J)_\kappa)\leq\alpha$ whenever
$\emptyset\neq J\in[I]^{<\kappa}$. Fix such $J$, for $i\in J$ let $D_i$
be dense in $X_i$ with $|D_i|=\beta_i<\alpha$, and set $D:=\Pi_{i\in
J}\,D_i$ and $\beta:=\sup_{i\in J}\,\beta_i$. Since
$|J|<\kappa<\alpha=\cf(\alpha)$ we have $\beta<\alpha$, and from
$\kappa\ll\alpha$ follows $|D|\leq\beta^{|J|}<\alpha$.
Clearly $D$ is dense in $(X_J)_\kappa$, and from
$d((X_J)_\kappa)<\alpha$ it then follows that
$S((X_J)_\kappa)\leq\alpha$, as required.

(b)~$\Rightarrow$~(a). Fix $\beta<\alpha$ and $\lambda<\kappa$, and set
$X:=(D(\beta))^\lambda$. Then $(X)_\kappa$ is discrete, and from
$S((X)_\kappa)\leq\alpha$ it follows that
$\beta^\lambda=|X|=|(X)_\kappa|<\alpha$.
\end{proof}

\begin{corollary}\label{S(Xkappa)}
Let $\alpha\geq2$ and $\kappa\geq\omega$, and let $\{X_i:i\in I\}$
be a set of spaces.

{\rm (a)} If $\alpha\geq\omega$ and $d(X_i)\leq2^\alpha$ for each $i\in
I$, then $S((X_I)_{\alpha^+})\leq(2^\alpha)^+$; and

{\rm (b)} if $d(X_i)\leq(\alpha^{<\kappa})^{<\kappa}$ for
each $i\in I$, then
$S((X_I)_\kappa)\leq((\alpha^{<\kappa})^{<\kappa})^+$.
\end{corollary}
\begin{proof} As noted in Remark~\ref{help}(3) and \ref{help}(4) we have
$\alpha^+\ll(2^\alpha)^+$ and
$\kappa\ll(\alpha^{<\kappa})^{<\kappa})^+$,
so Theorem~\ref{iffk<<a} applies
(with $(2^\alpha)^+$ replacing $\alpha$ in (a) and
with $((\alpha^{<\kappa})^{<\kappa})^+$ replacing $\alpha$ in (b)).
\end{proof}

The following result, which we are going to use frequently, shows a 
relationship between the Souslin number of product spaces of the type 
$(X_I)_\kappa$ and the Souslin numbers $S(X_i)$, $i\in I$.

\begin{lemma}\label{obs1}
Let $\alpha\geq3$ and $\kappa\ge\omega$ be cardinals,
and let $\{X_i:i\in I\}$ be a set of spaces such 
that $S(X_i)\geq\alpha$ for each $i\in I$. Let $\mu=\min\{\kappa,|I|^+\}$. 
Then $S((X_I)_{\kappa})>\beta^{<\mu}$ for each $\beta<\alpha$.
\end{lemma}
\begin{proof}
We show first that if $\kappa\ge|I|$ then $S((X_I)_\kappa)>\kappa$.
Indeed, in this case there is $J\in[I]^\kappa$, say 
$J=\{i_\eta:\eta<\kappa\}$. For $i_\eta\in J$ let $U(i_\eta,0)$ and 
$U(i_\eta,1)$ be nonempty, disjoint open subsets of $X_{i_\eta}$ and for 
$\eta<\kappa$ set $$U(\eta):=(\Pi_{\xi<\eta}\,U(i_\xi,0)\times
(U(i_\eta,1))\times(\Pi_{i\in I\backslash \{i_\xi:\xi\le\eta\}}\,X_i).$$
Then $\sC(\kappa):=\{U(\eta):\eta<\kappa\}$ is cellular in
$(X_I)_\kappa$, so $S((X_I)_\kappa)>|\sC(\kappa)|=\kappa$. 

Now fix $\beta<\alpha$, $\lambda<\mu$ and $J\in[I]^\lambda$. For $i\in J$
let $\{U(i,\eta):\eta<\beta\}$ be a cellular family in $X_i$, and for
$f\in\beta^J$ set $U(f):=(\Pi_{i\in J}\,U(i,f(i)))\times(X_{I\backslash J})$. 
Then $\sC:=\{U(f):f\in\beta^J\}$ is cellular in $(X_I)_{\kappa}$, and
\begin{equation}\label{EqN}
S((X_I)_{\kappa})>|\sC|=|\beta^J|=\beta^\lambda.
\end{equation}
Since $S((X_I)_\kappa)<\beta^\lambda$ for each $\lambda<\kappa$, we have
$$S((X_I)_{\kappa})\ge\beta^{<\mu} \mbox{ for each } \beta<\alpha.$$

To show that $S((X_I)_{\kappa})>\beta^{<\mu}$ for each $\beta<\alpha$ we 
consider three cases.

\underline{Case 1}.
The cardinal number $\beta^{<\mu}$ is singular. 
Then clearly $S((X_I)_{\kappa})>\beta^{<\mu}$. 

\underline{Case 2}.
The cardinal number $\beta^{<\mu}$ is regular
and there is $\nu<\mu$ such that $\beta^\nu = \beta^{<\mu}$. 
Then $S((X_I)_{\kappa})>\beta^\nu=\beta^{<\mu}$ by (\ref{EqN}).

\underline{Case 3}. Cases 1 and 2 fail.
Then, 
according to Lemma \ref{L4.13}(b), $\beta^{<\mu}=\mu$ and $\mu$ is a regular 
strong limit cardinal. Since $\mu=\min\{\kappa,|I|^+\}$, we have $\mu=\kappa$, 
hence $|I|\ge\kappa$ and since in that case $S((X_I)_{\kappa})>\kappa$ we 
conclude that $S((X_I)_{\kappa})>\beta^{<\mu}$.
\end{proof}

\begin{theorem}\label{Obsn}
Let $\kappa\ge\omega$ be a limit cardinal and let
$\{X_i:i\in I\}$ be a set of spaces such that $|I|\geq\kappa$ and
$S(X_i)\ge 3$ for each $i\in I$. Let also
$$\alpha:=\sup\{S((X_I)_\gamma):\gamma<\kappa\} \mbox{ for }\kappa>\omega
\mbox{ and }$$ 
$$\alpha:=\sup\{S(X_J):J\in[I]^{<\kappa}\} \mbox{ for }\kappa=\omega.$$
Then

{\rm (a)} $\kappa\leq\alpha\leq S((X_I)_\kappa)\leq\alpha^+$,
and $\kappa^+\leq S((X_I)_\kappa)$;

{\rm (b)} if $\alpha$ is regular and $\kappa<\alpha$ then
$S((X_I)_\kappa)=\alpha$; and

{\rm (c)} if $\alpha$ is singular or $\kappa=\alpha$
then $S((X_I)_\kappa)=\alpha^+$.
\end{theorem}
\begin{proof}
In each of (a), (b), or (c) the case $\kappa=\omega$ follows from Theorem 
\ref{firstSP}. Therefore below we consider only the case $\kappa>\omega$.

(a) That $\alpha\leq S((X_I)_\kappa)$ is obvious. Let 
$\mu=\min\{\kappa,|I|^+\}$. Since $|I|\ge\kappa$, we have $\mu=\kappa$. 
Then it follows from Lemma \ref{obs1} that 
$S((X_I)_{\kappa})>2^{<\mu}=2^{<\kappa}\ge\kappa$, hence 
$S((X_I)_{\kappa})\ge\kappa^+$. Also, since $\kappa > \omega$,
$S((X_I)_{\gamma^+})>2^{<\gamma^+}=2^{\gamma}\ge\gamma^+$ for every 
$\gamma<\kappa$, hence $\alpha\ge\kappa$. 

To prove $S((X_I)_\kappa)\leq\alpha^+$, suppose there is in $(X_I)_\kappa$ a 
basic cellular family $\sC$ such that $|\sC|=\alpha^+$, and for 
$\gamma<\kappa$ set $C(\gamma):=\{U\in\sC:|R(U)|<\gamma\}$. Then since 
$\alpha^+$ is regular with $\alpha^+>\kappa$, there is $\gamma<\kappa$ such 
that $|\sC(\gamma)|=\alpha^+$ and we have the contradiction $\alpha\geq 
S((X_I)_\gamma)\geq\alpha^{++}$.

(b) A similar argument applies. If in $(X_I)_\kappa$ there is a
basic cellular family $\sC$ such that $|\sC|=\alpha$ then with
$$\sC(\gamma):=\{U\in\sC:|R(U)|<\gamma\} \mbox{ for } \gamma<\kappa$$
we have $\sC=\bigcup_{\gamma<\kappa}\,\sC(\gamma)$ and from
the regularity of $\alpha$ and the relation $\kappa<\alpha$ we have
$|\sC(\gamma)|=\alpha$ for some $\gamma<\kappa$ and then
$$\alpha\geq S((X_I))_\kappa)\geq S((X_I)_\gamma)>\alpha,$$
a contradiction.

(c) Since $S((X_I)_\kappa)$ is regular, this is immediate from (a).
\end{proof}

\begin{theorem}\label{T4.12}
Let $\alpha\geq2$ and $\kappa\geq\omega$ be cardinals and
let $\{X_i:i\in I\}$
be a set of spaces with $d(X_i)\leq\alpha$ for each $i\in I$.
Then $S((X_I)_\kappa)\leq(\alpha^{<\kappa})^+$.
\end{theorem}

\begin{proof}
We assume first that $\alpha\geq\omega$ and we consider two cases.

\underline{Case 1}. $\alpha^{<\kappa}=(\alpha^{<\kappa})^{<\kappa}$.
The conclusion is immediate from 
Corollary~\ref{S(Xkappa)}
(even with the hypothesis $d(X_i)\leq\alpha$ weakened to
$d(X_i)\leq\alpha^{<\kappa}=(\alpha^{<\kappa})^{<\kappa}$). 

\underline{Case 2}. Case 1 fails. Then $\kappa$ is singular (by
Theorem~\ref{weak<kappa}(a)) and therefore a limit cardinal such that
$\kappa>\omega$. 
If there exists $\gamma<\kappa$ such that 
$S((X_I)_\kappa)=S((X_I)_{\gamma^+})$ then since $\gamma^+$ is regular
it follows from Corollary~\ref{S(Xkappa)}(b) that 
$$S((X_I)_\kappa)=S((X_I)_{\gamma^+})\le(\alpha^{\gamma})^{\gamma}\le\alpha^{<\kappa}<(\alpha^{<\kappa})^+.$$
If there is no $\gamma<\kappa$ such that 
$S((X_I)_\kappa)=S((X_I)_{\gamma^+})$ then
for each $\gamma<\kappa$ we have 
$S((X_I)_{\gamma^+})\le(\alpha^{\gamma})^{\gamma}\le\alpha^{<\kappa}$
and hence $$S((X_I)_\kappa)\leq(\sup_{\gamma<\kappa}S((X_I)_\gamma))^+\le(\alpha^{<\kappa})^+$$
from Theorem~\ref{Obsn}(a).

It remains to consider the case $\alpha<\omega$.
Note that $d(X_i)\leq\omega$ for each $i\in I$. Then if $\kappa=\omega$
we have
$$S((X_I)_\kappa)=S(X_I)\leq\omega^+=(\alpha^{<\kappa})^+$$
from Theorem~\ref{secondSP}, and
if $\kappa>\omega$ then the preceding paragraphs
apply to give
$$S((X_I)_\kappa)\leq(\omega^{<\kappa})^+\leq((2^\omega)^{<\kappa})^+
=(2^{<\kappa})^+\leq(\alpha^{<\kappa})^+.$$
\vskip-18pt
\end{proof}

\begin{remark}
{\rm (a) If the hypothesis $d(X_i)\leq\alpha$ of Theorem~\ref{T4.12} is weakened
to $d(X_i)\leq\alpha^{<\kappa}$, the conclusion can fail. To see that, it
is enough to refer to Discussion~\ref{gitik-shelah}, where we noted that
for every pre-assigned $\alpha\geq\omega$ the choice
$\kappa:=\beth_\lambda(\alpha)$ with $\lambda\leq\cf(\alpha)$ guarantees
that the space $E:=(D(\alpha^{<\kappa}))^I$ with $|I|=\alpha$ has
$E_\kappa$ discrete (since $\kappa>|I|$) and 
$|E_\kappa|=(\alpha^{<\kappa})^\alpha=\kappa^\alpha>\alpha^{<\kappa}$,
hence $S(E_\kappa)=|E_\kappa|^+>(\alpha^{<\kappa})^+$.

(b) With Theorem \ref{T4.12} in hand the implication (a)$\Rightarrow$(b) in 
Theorem \ref{iffk<<a} becomes now a direct corollary. Indeed, if
$\alpha=\beta^+$ in 
Theorem \ref{iffk<<a} then $d(X_i)\le\beta$ and according to 
Theorem \ref{T4.12} we have 
$$S((X_I)_\kappa)\le(\beta^{<\kappa})^+=(\Sigma_{\lambda<\kappa}\,
\beta^\lambda)^+\le(\beta\cdot\kappa)^+=\alpha$$ since $\kappa\ll\alpha$.
And if $\alpha$ is a regular limit cardinal in Theorem \ref{iffk<<a}
then Theorem \ref{T4.12} gives 
$S((X_I)_\kappa)\leq(\alpha^{<\kappa})^+$. But in this case, since $\alpha$ is 
regular and no set in $[\alpha]^{<\kappa}$ is cofinal in $\alpha$ we have 
$$\alpha^{<\kappa}=|[\alpha]^{<\kappa}|\leq\Sigma_{\eta<\alpha}\,
|[\eta]^{<\kappa}|=\Sigma_{\eta<\alpha}\,|\eta|^{<\kappa}\leq\Sigma_{\eta<\alpha}
\,\alpha=\alpha,$$
since $|\eta|^{<\kappa}=\Sigma_{\zeta<\kappa}\,|\eta|^\zeta\le\kappa\cdot\alpha=\alpha$ for every $\eta<\alpha$ whenever 
$\kappa\ll\alpha$.
}
\end{remark}

Theorem \ref{largeS}, using some of those same ideas, strengthens that result.
For use in its proof and frequently thereafter
we adopt henceforth the following notational convention
concerning limit cardinals $\kappa$.
We do not exclude here the possibility that $\kappa$
is regular, but this convention will be invoked chiefly in cases where it is
known that $\cf(\kappa)<\kappa$.

\begin{notation}\label{defkappaeta}
{\rm
Let $\kappa\geq\omega$ be a limit cardinal. Then
$\{\kappa_\eta:\eta<\cf(\kappa)\}$ is a set of cardinals such that

(a) $\kappa_\eta<\kappa_{\eta'}<\kappa \mbox{ when
}\eta<\eta'<\cf(\kappa)$, and
 
(b) $\Sigma_{\eta<\cf(\kappa)}\,\kappa_\eta=\kappa.$
}
\end{notation}

\begin{theorem}\label{largeS}
Let $\alpha\geq2$, let $\kappa>\omega$ be a
(possibly regular) limit cardinal, and let
$\{\kappa_\eta:\eta<\cf(\kappa)\}$ be a family of cardinals as in
Notation~\ref{defkappaeta}. 
For $\eta<\cf(\kappa)$ let $\{X(\eta):\eta<\cf(\kappa)\}$ be a 
(not necessarily faithfully indexed) set of spaces 
such that $S(X(\eta))\geq\alpha^{\kappa_\eta}$ for each $\eta<\cf(\kappa)$, and 
let $X:=\Pi_{\eta<\cf(\kappa)}\,X(\eta)$. Then

{\rm (a)} $S(X_{\cf(\kappa)})\geq\alpha^{<\kappa}$; and

{\rm(b)} if $\alpha^{<\kappa}<(\alpha^{<\kappa})^{<\kappa}$, then
$S(X_{(\cf(\kappa))^+})>\alpha^\kappa\geq(\alpha^{<\kappa})^+$.
\end{theorem}
\begin{proof}
(a) is obvious, since $S(X_{\cf(\kappa)})\geq 
S(X(\eta))\geq\alpha^{\kappa_\eta}$ for each $\eta<\cf(\kappa)$.

(b) The topology of $X_{(\cf(\kappa))^+}$ is the (full) box topology. 
Since $\cf(\alpha^{<\kappa})=\cf(\kappa)<\kappa$ by 
Theorem~\ref{weak<kappa}, 
we may assume without loss of generality that
$\alpha^{\kappa_\eta}<\alpha^{\kappa_{\eta'}}$ for
$\eta<\eta'<\cf(\kappa)$. 
Let $\sC(\eta):=\{X(\eta)\}$ for limit ordinals $\eta<\cf(\kappa)$, and
for $\eta<\cf(\kappa)$ let $\sC(\eta+1)$ be cellular in
$X(\eta+1)$ with $|\sC(\eta+1)|\ge\alpha^{\kappa_{\eta}}$. Then
$\sC:=\{\Pi_{\eta<\cf(\kappa)}\,C_\eta:C_\eta\in\sC(\eta)\}$
is cellular in $X_{(\cf(\kappa))^+}$, with
$$|\sC|=\Pi_{\eta<\cf(\kappa)}\,|\sC(\eta)|
\geq\Pi_{\eta<\cf(\kappa)}\,\alpha^{\kappa_\eta}
=\alpha^{\Sigma_{\eta<\cf(\kappa)}\,\kappa_\eta}
=\alpha^\kappa,$$
\noindent so $S(X_{(\cf(\kappa))^+})>\alpha^\kappa\geq(\alpha^{<\kappa})^+$.
\end{proof}

The following simple lemma, strictly set-theoretic (non-topological) in
nature, is one of several preliminaries required for the proof of
Theorem~\ref{T4.15}.

\begin{lemma}\label{L4.13}
Let $\alpha\geq2$ and $\kappa\geq\omega$ be cardinals.

{\rm (a)} If $\alpha^{<\kappa}$ is a successor cardinal then there is 
$\nu<\kappa$ such that $\alpha^\nu = \alpha^{<\kappa}$.

{\rm (b)} If $\alpha^{<\kappa}$ is a regular cardinal and there is no 
$\nu<\kappa$ such that $\alpha^\nu = \alpha^{<\kappa}$ then 
$\alpha^{<\kappa}=\kappa$ and $\kappa$ is a regular strong limit 
cardinal.
\end{lemma}

\begin{proof}
(a) Let $\alpha^{<\kappa}=\lambda^+$. 

If $\kappa=\alpha^{<\kappa}$ then $\alpha^{<\kappa}=\alpha^\lambda$ (with
$\lambda<\kappa$).

If $\kappa<\alpha^{<\kappa}$ and $\alpha^\nu<\alpha^{<\kappa}$ for each
$\nu<\kappa$, then we have the contradiction
$$\alpha^{<\kappa}=\Sigma_{\nu<\kappa}\,\alpha^\nu\leq\lambda\cdot\kappa=
\lambda<\lambda^+=\alpha^{<\kappa}.$$

(b) It follows from (a) that if (b) fails and there is no
$\nu<\kappa$ such that $\alpha^\nu=\alpha^{<\kappa}$, then
$\alpha^{<\kappa}$ is a (regular) limit cardinal and we have
$$\alpha^{<\kappa}=\cf(\alpha^{<\kappa})\leq\cf(\kappa)
\leq\kappa\leq\alpha^{<\kappa}.$$
Hence $\alpha^{<\kappa}=\kappa$, and for each $\nu<\kappa$ we
have
$$2^\nu\leq\alpha^\nu<\alpha^{<\kappa}=\kappa,$$
as required.
\end{proof}

\begin{remark}\label{nonvac}
{\rm
It is not difficult to show, as in \cite[3.12]{comfneg74}, that for every
uncountable regular cardinal $\alpha$ there is a product space $X_I$
such that $S(X_I)=S((X_I)_\omega)=\alpha$; indeed, as noted there, with
$Y:=\Pi_{\beta<\alpha}\,D(\beta)$ one has $S(Y^I)=\alpha$ for all
nonempty sets $I$. Thus the instance $S(X_I)=\alpha$ allowed by
Theorem~\ref{firstSP} does in fact arise in non-trivial circumstances,
provided that uncountable regular limit cardinals $\alpha$ do exist.
In any case it is immediate from Theorem~\ref{firstSP} that for every
infinite cardinal $\alpha$ of the form $\alpha=\beta^+$ one has
$S((D(\beta))^I)=\alpha$ for all nonempty sets $I$.
The $\kappa$-box analogue of these statements
holds for suitable regular cardinals $\alpha$
(see Theorem~\ref{T4.15}(a) and Remark~\ref{condsi--iv}(b) below),
but the full analogue fails consistently (see Remark~\ref{G-Smodels}).

We continue with results preparatory to the proof of
Theorem~\ref{T4.15}.
}
\end{remark}

\begin{theorem}\label{T4.14}
Let $\alpha\ge 3$ and $\kappa\ge\omega$ be cardinals, and let
$\{X_i:i\in I\}$ be a set of spaces such that $|I|^+\geq\kappa$ and
$\alpha\leq S(X_i)$ for each $i\in I$. Then 

{\rm (a)} $\alpha^{<\kappa}\leq S((X_I)_\kappa);$ and

{\rm (b)} if in addition $\alpha<\kappa$ then 
$(2^{<\kappa})^+=(\alpha^{<\kappa})^+\le S((X_I)_\kappa).$
\end{theorem}

\begin{proof} 
(a) We consider two cases.

\underline{Case 1}. $\alpha$ is singular. Then for each $i\in I$ we have
$S(X_i)\ge\alpha^+$ and it follows from Lemma \ref{obs1} (with
$\alpha^+$ now replacing $\alpha$)
that for each $\lambda<\kappa$ we have $S((X_I)_\kappa)>\alpha^\lambda$. Thus
$$S((X_I)_\kappa)\geq\sup_{\lambda<\kappa}\,(\alpha^\lambda)^+\ge
\Sigma_{\lambda<\kappa}\,\alpha^\lambda=\alpha^{<\kappa}.$$

\underline{Case 2}. $\alpha$ is regular. (We consider here only the case 
$\alpha\ge\kappa$ since the case $\alpha<\kappa$ is considered in (b).) 
Fix $\beta<\alpha$. Since $S(X_i)>\beta$ for each $i\in I$, it follows from  
Lemma \ref{obs1} that $S((X_I)_\kappa)\ge(\beta^\lambda)^+$ for every 
$\lambda<\kappa$. Therefore 
\begin{equation}\label{Eq7}
S((X_I)_\kappa)\geq\sup_{\beta<\alpha}\,(\beta^\lambda)^+
\ge\Sigma_{\beta<\alpha}\,\beta^\lambda.
\end{equation}
Since $\alpha$ is regular and $\lambda<\kappa\leq\alpha$, for each
$A\in[\alpha]^\lambda$ there is $\xi<\alpha$ such that $A\subseteq\xi$ (with
$|\xi|<\alpha$), so $\alpha^\lambda=\Sigma_{\beta<\alpha}\,\beta^\lambda$. 
It follows from (\ref{Eq7}) that $S((X_I)_\kappa)\geq\alpha^\lambda$ for each
$\lambda<\kappa$. Hence $S((X_I)_\kappa)\geq\alpha^{<\kappa}$, as required.

(b) Since $$2^{<\kappa}\leq\alpha^{<\kappa}\leq(2^\alpha)^{<\kappa}=
\Sigma_{\lambda<\kappa}\,(2^\alpha)^\lambda=
\Sigma_{\lambda<\kappa}\,2^\lambda=2^{<\kappa},$$
we have $$2^{<\kappa}=\alpha^{<\kappa}.$$

Now, fix $\lambda<\kappa$. Since $S(X_i)>2$ for each $i\in I$, it follows from 
Lemma \ref{obs1} that 
\begin{equation}\label{Eq8new}
S((X_I)_\kappa)\geq S((X_I)_{\lambda^+})\geq(2^\lambda)^+.
\end{equation}

\underline{Case 1}. There exists $\nu < \kappa$ such that 
$\alpha^\nu = \alpha^{<\kappa}$. Since $\alpha<\kappa$, without 
loss of generality, we can assume that $\nu\ge \alpha$. Then 
$\alpha^\nu=2^\nu$ and from (\ref{Eq8new}) we get
$$S((X_I)_\kappa)\geq (2^\nu)^+ > 2^\nu=\alpha^{<\kappa}.$$ 

\underline{Case 2}. Case 1 fails. If $\alpha^{<\kappa}$ is regular 
then it follows from Lemma \ref{L4.13}(b) that $\kappa=\alpha^{<\kappa}$ 
and $\kappa$ is a regular strong limit cardinal.
If $\kappa=\omega$ then surely $S((X_I)_\kappa)\geq\kappa^+$, and if
$\kappa>\omega$ then
Theorem \ref{Obsn}(a) applies to give
$$S((X_I)_\kappa)\ge\kappa^+=(\alpha^{<\kappa})^+.$$
Now let $\alpha^{<\kappa}$ be singular. Since (\ref{Eq8new}) holds for 
every $\lambda<\kappa$ we have 
$$S((X_I)_\kappa)\geq\sup_{\lambda<\kappa}\,(2^\lambda)^+
\ge\Sigma_{\lambda<\kappa}\,2^\lambda=\alpha^{<\kappa}$$ and since 
$\alpha^{<\kappa}$ is singular we have 
$S((X_I)_\kappa)\geq(\alpha^{<\kappa})^+$, as required.
\end{proof}

\begin{corollary}
Let $\alpha$, $\beta$ and $\kappa$ be cardinals with $\alpha\ge 3$ and $\kappa\ge\omega$, and let
$\{X_i:i\in I\}$ be a set of spaces such that $|I|^+\geq\kappa$, and
$d(X_i)\leq \beta$ and $\alpha\leq S(X_i)$ for each $i\in I$. Then
$\alpha^{<\kappa}\leq S((X_I)_\kappa)\leq(\beta^{<\kappa})^+.$
\end{corollary}

\begin{proof}
Follows directly from Theorem~\ref{T4.12} and Theorem~\ref{T4.14}.
\end{proof}

\begin{corollary}\label{L4.14}
Let $\alpha\geq3$ and $\kappa\geq\omega$ be cardinals, and let
$\{X_i:i\in I\}$ be a set of spaces such that $|I|^+\geq\kappa$ and
$d(X_i)\leq\alpha\leq S(X_i)$ for each $i\in I$. Then
$\alpha^{<\kappa}\leq S((X_I)_\kappa)\leq(\alpha^{<\kappa})^+.$
\end{corollary}

Corollary~\ref{L4.14} provides tight parameters, but leaves undetermined
the question exactly when it is that the value of $S((X_I)_\kappa)$ is
$\alpha^{<\kappa}$ and when it is $(\alpha^{<\kappa})^+$. In the following 
theorem we settle that matter completely.

\begin{theorem}\label{T4.15}
Let $\alpha\geq3$ and $\kappa\geq\omega$ be cardinals, and let
$\{X_i:i\in I\}$ be a set of spaces such that $|I|^+\geq\kappa$ and  
$d(X_i)\leq\alpha\leq S(X_i)$ for each $i\in I$. Consider these conditions: {\rm (i)}~$\alpha$ is
regular; {\rm (ii)}~$\alpha=\alpha^{<\kappa}$; {\rm (iii)}
$\kappa\ll\alpha$; {\rm (iv)}~$S((X_J)_\kappa)=\alpha$ for all nonempty
$J\in[I]^{<\kappa}$. Then:

{\rm (a)} if conditions {\rm (i), (ii), (iii)} and {\rm (iv)} hold, then
$S((X_I)_\kappa)=\alpha^{<\kappa}=\alpha$; and

{\rm (b)} if one (or more) of conditions {\rm (i), (ii), (iii)} or
{\rm (iv)} fails, then
$S((X_I)_\kappa)=(\alpha^{<\kappa})^+$.
\end{theorem}
\begin{proof} (a) is immediate from Theorem~\ref{CN}, since
$S(X_i)=\alpha=\alpha^{<\kappa}$ for each $i\in I$ under the present
hypotheses.

(b) It suffices, according to Corollary~\ref{L4.14}, to
assume that $S((X_I)_\kappa)=\alpha^{<\kappa}$
and to show that conditions (i), (ii), (iii) and (iv) must hold. We consider
two cases.

\underline{Case 1}. There is $\nu<\kappa$ such that
$\alpha^\nu=\alpha^{<\kappa}$. We fix such $\nu$.

If (i) fails then $S(X_i)=\alpha^+$ and from Lemma \ref{obs1}
(with $\alpha^+$ and $\nu$ in the roles of $\alpha$ and $\mu$,
respectively) we
have $S((X_I)_\kappa)\geq S((X_I)_{\nu^+})>\alpha^\nu=\alpha^{<\kappa}$, a 
contradiction. Thus (i) holds.

To see that (ii) holds, suppose first that there are
$\beta<\alpha$ and $\lambda<\kappa$ such that $\beta^\lambda\geq\alpha$.
Then
$$\alpha^{<\kappa}=\alpha^\nu\leq\beta^{\lambda\cdot\nu}\leq\alpha^{\lambda\cdot\nu}\leq\alpha^{<\kappa}$$
so from Lemma~\ref{obs1} we conclude that 
$S((X_I)_\kappa)>\beta^{<\kappa}\geq\beta^{\lambda\cdot\nu}=\alpha^{<\kappa}$, 
a contradiction. Thus
\begin{equation}\label{Eq8}
\beta^\lambda<\alpha \mbox{ for all } \beta<\alpha, \lambda<\kappa.
\end{equation}
It follows that $\kappa\leq\alpha$.
Then each $\lambda<\kappa$ satisfies $\lambda<\alpha$ and from the regularity
of $\alpha$ we have
$[\alpha]^\lambda=\bigcup_{\beta<\alpha}\,[\beta]^\lambda$
for each such $\lambda$. Thus (\ref{Eq8}) gives
$$\alpha^{<\kappa}=\Sigma_{\lambda<\kappa}\,\alpha^\lambda\leq
\Sigma_{\lambda<\kappa}[\Sigma_{\beta<\alpha}\,\beta^\lambda]
\leq\kappa\cdot\alpha\cdot\alpha=\alpha\leq\alpha^{<\kappa},$$
and (ii) is proved. To show (iii) we need only show
$\kappa<\alpha$, since (\ref{Eq8}) then gives $\kappa\ll\alpha$. Suppose then 
that $\kappa=\alpha$. Then (\ref{Eq8}) shows that $\kappa$ is a (regular, 
strong) limit cardinal, so from Theorem \ref{Obsn} we have the contradiction
$S((X_I)_\kappa)>\kappa=\alpha=\alpha^{<\kappa}$.
Thus $\kappa<\alpha$ and the proof of (iii) is complete. To prove
(iv), it suffices to note that if
$S((X_J)_\kappa)>\alpha$ for some nonempty
$J\in[I]^{<\kappa}$, then we have the contradiction
$S((X_I)_\kappa)>\alpha=\alpha^{<\kappa}$.

\underline{Case 2}. There is no $\nu<\kappa$ such that
$\alpha^\nu=\alpha^{<\kappa}$. If $\alpha^{<\kappa}$ is singular then
$S((X_I)_\kappa)=\alpha^{<\kappa}$ is impossible, so $\alpha^{<\kappa}$
is regular and
Lemma~\ref{L4.13}(b) applies to show that $\alpha^{<\kappa}=\kappa$ is a
(regular, strong) limit cardinal; from Theorem \ref{Obsn} we again have
the contradiction $S((X_I)_\kappa)>\kappa=\alpha^{<\kappa}$.
\end{proof}
Although every infinite Souslin number is regular and uncountable, hence
is either a successor cardinal or an uncountable regular limit cardinal,
it is perhaps not clear from Theorem~\ref{T4.15} exactly which
uncountable regular cardinals occur in the form $S((X_I)_\kappa)$ with
$\{X_i:i\in I\}$ constrained as in Theorem~\ref{T4.15}. Is part (a) of that
theorem potentially vacuous? Can every successor cardinal
$\beta^+$ occur as $\beta^+=\alpha$ in Theorem~\ref{T4.15}(a)? For each
$\kappa$, can some $\beta^+=\alpha$ so occur? Do there exist, for every
regular limit cardinal $\alpha$, infinite $\kappa\ll\alpha$ and spaces
$\{X_i:i\in I\}$ such that $S((X_I)_\kappa)=\alpha$? We address these
questions in \ref{condsi--iv}---\ref{S=lim} below.
\begin{remarks}\label{condsi--iv}
{\rm
(a) Let $\beta$ be a singular cardinal and set $\alpha:=\beta^+$. Let
$I$ be an uncountable set and for $i\in I$ set $X_i:=D(\beta)$. Clearly
(i), (ii) and (iii) are satisfied with $\kappa=\omega$; also (iv) is
satisfied with $\kappa=\omega$, since if $J\in[I]^{<\omega}$ then
$X_J=X^J$ is discrete with $|X_J|=\beta$, so
$S(X_I)=S((X_I)_\omega)=\beta^+=\alpha$. Thus
$S(X_I)=S((X_I)_\omega)=\alpha$ by Theorem~\ref{T4.15}(a). The same
conclusion is available from Theorem~\ref{T4.15}(b) by replacing
$\alpha$ everywhere in the statement of Theorem~\ref{T4.15} by $\beta$.
In this case both (i) and (iv) fail for $\beta$, so
$S(X_I)=S((X_I)_\omega)=\beta^+=\alpha$ by Theorem~\ref{T4.15}(b).

(b) Similar examples exist in ZFC for every
uncountable regular cardinal $\kappa$.
Indeed, given such $\kappa$ let
$\gamma\geq2$ be arbitrary
and set $\beta:=\beth_\kappa(\gamma)$ and $\alpha:=\beta^+$.
For $\lambda<\kappa$ we have
$$[\beta]^\lambda=\bigcup_{\delta<\beta}\,[\delta]^\lambda \mbox{ and }
[\alpha]^\lambda=\bigcup_{\xi<\alpha}\,[\xi]^\lambda,$$
so $$\beta^\lambda\leq\Sigma_{\delta<\beta}\,\delta^\lambda
\leq\Sigma_{\delta<\beta}\,2^\delta\cdot2^\lambda=\beta$$
and $$\alpha^\lambda\le\alpha\cdot\beta^\lambda\leq\alpha\cdot\beta=\alpha.$$
Conditions (i), (ii) and (iii) are then clear
and again, as in (a), if $|I|>\kappa$ and $X_i:=D(\beta)$ for each $i\in
I$, then each space $(X^J)_\kappa=(X_J)_\kappa$ is discrete (when
$|J|<\kappa$) with
$|X_J|=\beta$, so $S((X_J)_\kappa)=\beta^+=\alpha$
and (iv) holds by Theorem~\ref{T4.15}(a).
Also as in (a) above the same
conclusion is available from Theorem~\ref{T4.15}(b) by replacing
$\alpha$ everywhere in the statement of Theorem~\ref{T4.15} by $\beta$.
In this case both (i) and (iv) fail for $\beta$, so
$S((X_I)_\kappa)=\beta^+=\alpha$ by Theorem~\ref{T4.15}(b).

(c) (a) and (b) above indicate that in all models of ZFC conditions (i),
(ii), (iii) and (iv) of Theorem~\ref{T4.15}(a) are satisfied by suitably
chosen cardinals and spaces, so part (a) of Theorem~\ref{T4.15} is not
vacuous. Those examples depend, however, on choosing for $\alpha$ a
regular cardinal of the form $\alpha=\beta^+$.
Part (a) of Theorem~\ref{T4.20} shows exactly which successor cardinals
$\gamma^+$ arise as $S((X_I)_\kappa)$ in Theorem~\ref{T4.15}, and part
(b) indicates when it can occur that $S((X_I)_\kappa)$ is a limit cardinal.
}
\end{remarks}
\begin{theorem}\label{T4.20}
Let $\alpha\geq3$ and $\kappa\geq\omega$ be cardinals, and let
$\{X_i:i\in I\}$ be a set of spaces such that $|I|^+\geq\kappa$ and
$d(X_i)\leq\alpha\leq S(X_i)$ for each $i\in I$. 

{\rm (a)} If $S((X_I)_\kappa)$ is a successor cardinal---say
$S((X_I)_\kappa)=\gamma^+$---then either conditions {\rm (i), (ii), (iii)} 
and {\rm (iv)} of Theorem \ref{T4.15} hold and
$\gamma=\gamma^{<\kappa}$, or at least one of those
conditions fails and $\gamma=\alpha^{<\kappa}$.

{\rm (b)} If $S((X_I)_\kappa)$ is a (regular) limit cardinal, then
$S((X_I)_\kappa)=\alpha$ and $d(X_i)=S(X_i)=\alpha$ for each
$i\in I$.
\end{theorem}
\begin{proof}
(a) If $\gamma^+=S((X_I)_\kappa)\neq(\alpha^{<\kappa})^+$
then by Theorem~\ref{T4.15}
the indicated conditions (i), (ii), (iii) and (iv) all hold and
$\gamma^+=S((X_I)_\kappa)=\alpha=\alpha^{<\kappa}$. Since
$\kappa\ll\alpha=\gamma^+$ we have $\gamma^\lambda=\gamma$ for all
$\lambda<\kappa$ and hence
$\gamma^{<\kappa}\leq\kappa\cdot\gamma=\gamma$,
so $\gamma=\gamma^{<\kappa}$ as asserted.

(b) If $S((X_I)_\kappa)$ is a regular limit cardinal then conditions
(i), (ii), (iii) and (iv) of Theorem~\ref{T4.15} all hold, so
$S((X_I)_\kappa)=\alpha$ by
Theorem~\ref{T4.15}(a); further for each $i\in I$ we have
$S(X_i)=\alpha$ by condition (iv). If there is $i\in I$ such
that $d(X_i)<\alpha$ then $(d(X_I))^+<\alpha=S(X_i)$, which is impossible.
\end{proof}

\begin{remark}\label{S=lim}
{\rm
It is well known that the
existence of an uncountable regular strong limit cardinal cannot be
established in ZFC \cite[12.12]{jech}, but it
should be noted that in case such a cardinal $\alpha$ exists then 
there are cardinals $\kappa$ and spaces $X_i$ to which
Theorem~\ref{T4.15}(a) and Theorem~\ref{T4.20} apply. Indeed,
let $\alpha$ be a regular limit cardinal and suppose that $\kappa$
satisfies $\omega\leq\kappa\ll\alpha=\alpha^{<\kappa}$. (These latter
conditions are satisfied by every infinite $\kappa<\alpha$, in case
$\alpha$ is in addition assumed to be a strong limit cardinal.) Let $I$
be a nonempty set and for $\beta<\alpha$ and $i\in I$ set
$D(\beta,i):=D(\beta)$. Then $\{D(\beta,i):\beta<\alpha,i\in I\}$ is
a set of spaces with each $d(D(\beta,i))=\beta<\alpha$, so by
Theorem~\ref{iffk<<a} the space
$Y:=\Pi_{\beta<\alpha, i\in I}\,D(\beta,i)$ satisfies
$S((Y)_\kappa)=\alpha$. As a set we have
$Y=X^I$ with $X:=\Pi_{\beta<\alpha}\,D(\beta)$, and the topology of the
space $Y_\kappa$ is finer than the topology of the space
$(X^I)_\kappa$, so also the power space $X^I$ satisfies
$S((X^I)_\kappa)=\alpha$, where $\alpha$ is a (regular strong) limit cardinal. 
Clearly $d(X)=\alpha$ and $S(X)=\alpha$ and 
therefore $(X^I)_\kappa$ is an example of a product space that satisfies all 
the hypotheses and the conclusion of Theorem \ref{T4.15}(a).
}
\end{remark}

\begin{lemma}\label{SLC}
Let $\kappa$ be a strong limit cardinal.

{\rm (a)} If $2\leq\alpha<\kappa$ then $\alpha^{<\kappa}=\kappa$.

{\rm (b)} If $\kappa$ is regular then $\kappa^{<\kappa}=\kappa$.

{\rm (c)} If $\kappa$ is singular then $\kappa^{<\kappa}=2^\kappa$.
\end{lemma}
\begin{proof} (a) We have
$$\kappa\leq\alpha^{<\kappa}\leq(2^\alpha)^{<\kappa}=\Sigma_{\lambda<\kappa}\,(2^\alpha)^\lambda
=\Sigma_{\lambda<\kappa}\,2^\lambda\leq\kappa\cdot\kappa=\kappa.$$

(b) This is proved in Lemma~\ref{regularSLC}(a).

(c) With $\{\kappa_\eta:\eta<\cf(\kappa)\}$ chosen as in 
Notation~\ref{defkappaeta} we have
$$2^\kappa=2^{\Sigma_{\eta<\cf(\kappa)}\,\kappa_\eta}=\Pi_{\eta<\cf(\kappa)}\,2^{\kappa_\eta}
\leq\kappa^{cf(\kappa)}\leq\kappa^{<\kappa}\leq2^\kappa.$$
\vskip-18pt
\end{proof}

\begin{theorem}\label{Spower}
Let $\kappa$ be a strong limit cardinal and $I$ be an index set with 
$|I|\ge\kappa$.

{\rm (a)} If $2\leq\alpha<\kappa$ then 
$S(((D(\alpha))^I)_\kappa)=\kappa^+$.

{\rm (b)} If $\kappa$ is regular then
$S(((D(\kappa))^I)_\kappa)=\kappa^+$.

{\rm (c)} If $\kappa$ is singular then
$S(((D(\kappa))^I)_\kappa)=(2^\kappa)^+$.
\end{theorem}
\begin{proof}
In each case, condition (iii) of
Theorem~\ref{T4.15} fails, so
parts (a), (b) and (c) follow from Theorem~\ref{T4.15}(b) and
from parts (a), (b) and (c) of Lemma~\ref{SLC}, respectively.
\end{proof}

\begin{remark}\label{S(G-S)}
{\rm
As we noted in Discussion \ref{gitik-shelah}(f) the value
of $d((\two^{\aleph_\omega})_{\aleph_\omega}$ depends on the model of 
{\rm ZFC}, while the findings we have enunciated here are sufficiently 
powerful that the weight and Souslin number of such spaces as 
$(\two^{\aleph_\omega})_{\aleph_1}$ and
$(\two^{\aleph_\omega})_{\aleph_\omega}$ in $\VV_1$ and $\VV_2$
now emerge painlessly. To make those computations,
recall that $2^{\aleph_n}=\aleph_{n+1}$ and
$2^{\aleph_\omega}=\aleph_\omega^\omega=\aleph_{\omega+2}$ there, so we have
from Theorem~\ref{T2.7}(b)
$$w((\two^{\aleph_\omega})_{\aleph_1})=\aleph_\omega^\omega=\aleph_{\omega+2},$$ also
$$w((\two^{\aleph_\omega})_{\aleph_\omega})=(\aleph_\omega)^{<\aleph_\omega}=\aleph_{\omega+2}$$
in both those models.

Concerning the Souslin number, it is clear that
$S((\two^{\aleph_0})_{\aleph_1})=(2^{\aleph_0})^+=\cc^+$ in
ZFC, so from Corollary~\ref{S(Xkappa)} we have for each nonempty set $I$ that 
$$\cc^+=S((\two^{\aleph_0})_{\aleph_1})\leq
S((\two^I)_{\aleph_1})\leq\cc^+,$$
hence $S((\two^{\aleph_\omega})_{\aleph_1})=\cc^+$ in ZFC
(with $\cc^+=\aleph_2$ in the models $\VV_1$ and $\VV_2$).

Finally, from Theorem~\ref{Spower}(a) (or from Theorem \ref{T4.15}(b)) we 
have 
$$S((\two^{\aleph_\omega})_{\aleph_\omega})=(\aleph_\omega)^+=\aleph_{\omega+1}$$ in $\VV_1$ and $\VV_2$.
}
\end{remark}

In Theorem \ref{GKurepa} below we give an upper bound for the Souslin number 
of a product space with the $\kappa$-box topology that depends only on the 
Souslin numbers of its coordinate spaces, rather than (as in 
Theorem \ref{GCN}) on the Souslin numbers of its ``small" sub-products. For 
that we need the following notation (see \cite{comfneg74}, \cite{comfneg82}, 
\cite{juh}).

\begin{notation}
{\rm
Let $\alpha$, $\beta$, $\kappa$ and $\lambda$ be cardinals. The 
\emph{arrow notation} $\alpha\rightarrow(\kappa)^\beta_\lambda$ denotes the 
following partition relation: if $[\alpha]^\beta=\cup_{i<\lambda}\,P_i$ then 
there are $A\subseteq\alpha$ and $i<\lambda$ such that $|A|=\kappa$ and 
$[A]^\beta\subseteq P_i$.
}
\end{notation}

Preliminary to Theorem~\ref{GKurepa} we give a combinatorial
lemma which makes plain the relevance of the arrow
relation
$\alpha\rightarrow(\kappa)^2_\lambda$ to numbers of the form
$S((X_I)_{\lambda^+})$. The (general) proof we give is as anticipated
in \cite[page~73]{comfneg74}; it parallels in all its
essentials that of the special case treated in \cite[3.13]{comfneg74}.

We remark that results significantly stronger than that of 
Lemma~\ref{arrow}, which have perhaps not received the attention or the 
recognition they deserve, were developed by Negrepontis and his school 
in Athens in the 1970's. It is shown in \cite[5.17]{comfneg82} for 
example, using the hypothesis
$\omega\leq\lambda<\kappa\ll\alpha$ with $\kappa$ and $\alpha$ regular,
that if each $S(X_i)\leq\kappa$ then not only is
$S((X_I)_{\lambda^+})\leq\alpha$, as in Lemma~\ref{arrow}, but in fact
of every $\alpha$-many nonempty open subsets of $(X_I)_{\lambda^+}$
some $\alpha$-many have the finite intersection property.

\begin{lemma}\label{arrow}
Let $\alpha$, $\kappa$ and $\lambda$ be infinite cardinals such that
$\alpha\rightarrow(\kappa)^2_\lambda$, and let $\{X_i:i\in I\}$ be a set
of spaces such that $S(X_i)\leq\kappa$ for each $i\in I$. Then
$S((X_I)_{\lambda^+})\leq\alpha$.\end{lemma}
\begin{proof}
Suppose that there is a faithfully indexed cellular family 
$\{U^\xi:\xi<\alpha\}$ of basic open subsets of $(X_I)_{\lambda^+}$, and 
for $\{\xi,\xi'\}\in[\alpha]^2$
let $i(\xi,\xi')\in I$ be such that 
$U^\xi_{i(\xi,\xi')}\cap U^{\xi'}_{i(\xi,\xi')}=\emptyset$. 
For $\xi<\alpha$ we define
$$I(\xi):=\{i(\xi,\xi'):\xi'<\alpha \mbox{ and }\xi\ne\xi'\}.$$
Since $i(\xi,\xi')\in R(U^\xi)\cap R(U^{\xi'})$ for $\{\xi,\xi'\}\in[\alpha]^2$ 
we have $|I(\xi)|\le|R(U^\xi)|\leq\lambda$ for $\xi<\alpha$. Let 
$\{i_{\xi,\eta}:\eta<\lambda\}$ be an indexing of $I(\xi)$ for 
$\xi<\alpha$, and for
$(\eta,\eta')\in \lambda\times\lambda$ set
$$P_{\eta,\eta'}:=\{\{\xi,\xi'\}\in[\alpha]^2:\xi<\xi'\mbox{ and }
i_{\xi,\eta}=i_{\xi',\eta'}\}$$
(some of the sets $P_{\eta,\eta'}$ might be empty). Since
$$[\alpha]^2=\bigcup_{(\eta,\eta')\in\lambda\times\lambda}\,P_{\eta,\eta'}$$
and $\alpha\rightarrow(\kappa)^2_\lambda$, there are
$A\in[\alpha]^\kappa$ and
$(\overline{\eta},\overline{\eta'})\in\lambda\times\lambda$ such that
$[A]^2\subseteq P_{\overline{\eta},\overline{\eta'}}$.
Thus there is $\overline{i}\in I$ such that if $\{\xi,\xi'\}\in[A]^2$ and 
$\xi<\xi'$ then
$i(\xi,\xi')=i_{\xi,\overline{\eta}}=i_{\xi',\overline{\eta'}}=\overline{i}$ 
and hence $U^\xi_{\overline{i}}\cap U^{\xi'}_{\overline{i}}=\emptyset$.
It follows that $\{U^\xi_{\overline{i}}:\xi\in A\}$ is cellular in
$X_{\overline{i}}$ and we have the contradiction
$S(X_{\overline{i}})>|A|=\kappa$.
\end{proof}

The following theorem is \cite[Theorem 1.5(a)]{comfneg82}. It is noted in
\cite{comfneg82} that preliminary formulations of Theorem~\ref{CNT1.5}
appear (with different hypotheses) in Erd{\H{o}}s and
Rado~\cite[39(iii)]{erdosrado56} and Kurepa~\cite{kurepa59a}. The
important and motivating special case
$(2^\alpha)^+\rightarrow(\alpha^+)^2_\alpha$ of Theorem~\ref{CNT1.5}
appeared as early as 1942~\cite{erdos42}, while the seminal instance
$\kappa=\omega$, $\alpha=\omega^+$ of Theorem~\ref{GKurepa}(a) was
given by Kurepa~\cite{kurepa62} (see also \cite[Theorem 3.13
and remark on pp.~73-74]{comfneg74}).

\begin{theorem}\label{CNT1.5}
If $\omega\le \kappa \ll \alpha$ with $\alpha$ and $\kappa$ regular,  
then $\alpha\rightarrow(\kappa)^2_\lambda$ for all $\lambda < \kappa$.
\end{theorem}

\begin{theorem}\label{GKurepa}
Let $\alpha\geq2$ and $\kappa\geq\omega$ be cardinals and let $\{X_i:i\in I\}$ 
be a family of nonempty spaces such that $S(X_i)\le\alpha^+$ for each $i\in I$.

{\rm (a)} If $\alpha^+\ge\kappa$ then 
$S((X_I)_\kappa)\le(2^\alpha)^+$; and

{\rm (b)} if $\alpha^+\leq \kappa$ then 
$S((X_I)_\kappa)\le((\alpha^{<\kappa})^{<\kappa})^+$.
\end{theorem}
\begin{proof}
(a) Clearly $\alpha\geq\omega$ here, so 
Remark~\ref{help}(2) applies to give $\alpha^+\ll(2^\alpha)^+$. We then have
$$(2^\alpha)^+\rightarrow(\alpha^+)^2_\alpha$$ by Theorem~\ref{CNT1.5},
so $S((X_I)_{\alpha^+})\leq(2^\alpha)^+$ by Lemma~\ref{arrow} (with
$(2^\alpha)^+$, $\alpha^+$, and $\alpha$ in the role of $\alpha$,
$\kappa$, and $\lambda$ there). Thus surely
$S((X_I)_\kappa)\leq(2^\alpha)^+$ if $\kappa\leq\alpha^+$.

(b) We consider three cases.

\underline{Case 1}. $\kappa$ is singular. Then
$\kappa^+\ll((\alpha^{<\kappa})^{<\kappa})^+$
by Remark~\ref{help}(4), hence we have from Theorem~\ref{CNT1.5} that
$$((\alpha^{<\kappa})^{<\kappa})^+\rightarrow(\kappa^+)^2_\kappa.$$
Since $\alpha\leq\kappa$ we have $S(X_i)\leq\kappa^+$ for each $i\in I$,
so in fact even
$$S((X_I)_\kappa)\leq S((X_I)_{\kappa^+})\leq((\alpha^{<\kappa})^{<\kappa})^+$$
by Lemma~\ref{arrow} (with $((\alpha^{<\kappa})^{<\kappa})^+$,
$\kappa^+$, and $\kappa$ in the role of $\alpha$, $\kappa$, and $\lambda$
there). 

\underline{Case 2}. $\kappa$ is a successor cardinal, say
$\kappa=\lambda^+$. Since
$\lambda^+\ll(\alpha^{\lambda})^+$, by Theorem~\ref{CNT1.5} we have 
$$(\alpha^\lambda)^+\rightarrow(\lambda^+)^2_\lambda,$$  
so $$S((X_I)_{\kappa})\leq(\alpha^\lambda)^+=((\alpha^{<\kappa})^{<\kappa})^+$$
by Lemma~\ref{arrow} (with
$(\alpha^\lambda)^+$, $\lambda^+$, and $\lambda$ in the role of $\alpha$,
$\kappa$, and $\lambda$ there).

\underline{Case 3}. $\kappa$ is regular limit cardinal. Then it follows
from Remark \ref{help}(3) that
$\omega\le \kappa \ll ((\alpha^{<\kappa})^{<\kappa})^+$ and 
therefore, according to Theorem \ref{CNT1.5}, 
$((\alpha^{<\kappa})^{<\kappa})^+\rightarrow(\kappa)^2_\lambda$
for all $\lambda<\kappa$. Then since $S(X_i)\leq\kappa$
for each $i\in I$ we have
\begin{equation}\label{Eq9}
S((X_I)_{\lambda^+})\le ((\alpha^{<\kappa})^{<\kappa})^+ \mbox{ for all } \lambda<\kappa
\end{equation}
\noindent from Lemma~\ref{arrow} (with
$((\alpha^{<\kappa})^{<\kappa})^+$ in the role of $\alpha$ there).

Now suppose that $\sC$ is a cellular family in $(X_I)_\kappa$ of canonical
open sets such that $|\sC|=((\alpha^{<\kappa})^{<\kappa})^+$,
and for $\lambda<\kappa$ let
$\sC(\lambda):=\{U\in\sC:|R(U)|<\lambda\}$. Then
$\sC=\bigcup_{\lambda<\kappa}\,\sC(\lambda)$ with
$$\cf(((\alpha^{<\kappa})^{<\kappa})^+)=((\alpha^{<\kappa})^{<\kappa})^+\geq\kappa^+>\kappa,$$
so there is $\lambda<\kappa$ such that
$|\sC(\lambda)|=((\alpha^{<\kappa})^{<\kappa})^+$.
Then
$$S((X_I)_{\lambda^+})>|\sC(\lambda)|=((\alpha^{<\kappa})^{<\kappa})^+,$$
contrary to (\ref{Eq9}).
\end{proof}
\begin{remark}
{\rm
We note that in those cases of Theorem~\ref{GKurepa} to which both (a) 
and (b) apply, namely when $\kappa=\alpha^+$, the
upper bounds provided by the estimates in (a) and (b) coincide.
Indeed with $\kappa=\alpha^+$ we have
$$((\alpha^{<\kappa})^{<\kappa})^+=((\alpha^\alpha)^\alpha)^+=(2^\alpha)^+.$$
}
\end{remark}

Combining Theorem \ref{T4.14} and Theorem \ref{GKurepa} we obtain the 
following corollary.

\begin{corollary}\label{C4.37}
Let $\alpha\geq3$ and $\kappa\geq\omega$ be cardinals, and let
$\{X_i:i\in I\}$ be a set of spaces such that $|I|^+\geq\kappa$ and
$\alpha\le S(X_i)\le \alpha^+$ for each $i\in I$. 

{\rm (a)} If $\alpha^+\ge\kappa$ then 
$\alpha^{<\kappa}\le S((X_I)_\kappa)\le(2^\alpha)^+$; and

{\rm (b)} if $\alpha^+\leq \kappa$ then 
$(2^{<\kappa})^+\le S((X_I)_\kappa)\le((2^{<\kappa})^{<\kappa})^+$.
\end{corollary}

\begin{corollary}\label{C4.40}
Let $\alpha\geq3$ and $\kappa\geq\omega$ be cardinals, and let
$\{X_i:i\in I\}$ be a set of spaces such that $|I|^+\geq\kappa$ and
$3\le S(X_i)\le \alpha^+$ for each $i\in I$. If $\alpha^+\le\kappa$ then
$S((X_I)_\kappa)=(2^{<\kappa})^+$.
\end{corollary}

\begin{proof}
The case $3=\alpha<\kappa$ of Theorem~\ref{T4.14}(b) gives
$S((X_I)_\kappa\geq(2^{<\kappa})^+$.
 
If $\kappa$ is regular or there is $\nu<\kappa$ such that $2^\nu=2^{<\kappa}$
then $2^{<\kappa}=(2^{<\kappa})^{<\kappa}$ by Theorem~\ref{weak<kappa}(a) and 
the statement is immediate from Corollary \ref{C4.37}(b). 

Now we assume that $\kappa$ is singular and that $2^\nu<2^{<\kappa}$ for each 
$\nu<\kappa$ and we let $\{\kappa_\eta:\eta<\cf(\kappa)\}$ be a 
set of cardinals as in Notation~\ref{defkappaeta}. 
Suppose that $S((X_I)_\kappa)>(2^{<\kappa})^+$, and let $\sC$ be a
cellular family of basic
open sets in $(X_I)_\kappa$ with $|\sC|=(2^{<\kappa})^+$.
Then with
$$\sC(\eta):=\{C:C\in \sC \mbox{ and } C \mbox{ is open in }
(X_I)_{\kappa_\eta}\}$$
for $\eta<\cf(\kappa)$ we have
$\sC=\cup_{\eta<\cf(\kappa)}\,\sC(\eta)$, and since
$\cf((2^{<\kappa})^+)>\cf(\kappa)$ there is
$\eta<\cf(\kappa)$ such that $|\sC(\eta)| = (2^{<\kappa})^+$; for this
$\eta$ we have
\begin{equation}\label{Eq8'}
S((X_I)_{\kappa_\eta})>(2^{<\kappa})^+.
\end{equation}
Since $\kappa$ is singular we have from $\alpha^+\leq\kappa$ that
$\alpha<\kappa$, so there is $\eta'<\cf(\kappa)$ such that
$\alpha<\kappa_{\eta'}$; we take $\eta'\geq\eta$. Then from
Theorem~\ref{GKurepa}(b) with $\kappa_{\eta'}^+$ replacing $\kappa$ we have
$$S((X_I)_{\kappa_\eta})\leq
S((X_I)_{\kappa_{\eta'}^+})\leq(\alpha^{\kappa_{\eta'}})^+
\leq((2^\alpha)^{\kappa_{\eta'}})^+=(2^{\kappa_{\eta'}})^+<2^{<\kappa},$$ 
which contradicts (\ref{Eq8'}).
\end{proof}
\begin{remarks}\label{G-Smodels}
{\rm
(a) We noted in Remark~\ref{condsi--iv} that the conditions given in
Theorem~\ref{T4.15} are satisfied by many pairs $\kappa$, $\alpha$ of
cardinals and for many sets $\{X_i:i\in I\}$ of spaces; in particular
(see condition (iv) of Theorem~\ref{T4.15}) for
$\kappa\ll\alpha=\alpha^{<\kappa}$ there are spaces $X$ such that
$S((X^I)_\kappa)=\alpha$ for all nonempty index sets $I$. We note now
that consistently there are (regular) $\alpha$ and $\kappa$ for which
the relation $S((X^I)_\kappa)=\alpha$ holds for no space $X$ and
infinite index set $I$.
Indeed, let $\VV$ be one of the Gitik-Shelah models whose salient
cardinality properties are given in
Discussion \ref{gitik-shelah}(d) and let $\kappa=\aleph_1$ and 
$\alpha=\aleph_{\omega+1}$. Suppose there is a space $X$ 
such that $S(X)=\alpha$ and $S((X^I)_\kappa)=\alpha$ for some infinite set 
$I$. Then $S(X)=\alpha=\aleph_{\omega+1}>\aleph_{\omega}$, and it follows 
from Lemma~\ref{obs1} that
$S((X^I)_\kappa)>(\aleph_{\omega})^\omega=\aleph_{\omega+2}$, a contradiction.

(b) The upper bound $S(X\times X)=(2^\alpha)^+$ for spaces $X$ such that
$S(X)=\alpha^+$, allowed by Theorem~\ref{GKurepa}(a), is in fact
achieved for many $\alpha$ and $X$. This was first shown by Galvin and
Laver (cf.~\cite{galvin}) assuming $\alpha^+=2^\alpha$ (see
\cite[7.13]{comfneg82} for a treatment of the construction)
and by examples in ZFC by Todor\v{c}evi\'{c} \cite{todor85},
\cite{todori}, \cite{todorii}.
When $\alpha^+=2^\alpha$ this strict increase from
$S(X)$ to $S(X\times X)$ is minimal in the sense that
$$S(X\times X)=(2^\alpha)^+=(\alpha^+)^+=(S(X))^+.$$
\noindent We note in contrast that
in the models discussed in (a) there 
are spaces $X$
such that $S(X^I)\ge (S(X))^{++}$, for every infinite set $I$. 
For example, for any space $X$ in those models satisfying
$S(X)=\alpha=\aleph_{\omega+1}$ and with $\kappa=\aleph_1$  we have
$$((\aleph_{\omega+1})^+)^+=\aleph_{\omega+3}\leq S((X^I)_\kappa)
\le(2^\alpha)^+=(2^{\aleph_{\omega+1}})^+$$
in these models. Similarly Fleissner~\cite[Section~5]{fleissner78},
in suitably defined Cohen models of ZFC, constructs spaces $X$ for which
$S(X)=\omega^+=\aleph_1$ and 
$S(X\times X)=\aleph_{\omega+2}>\cc=\aleph_{\omega+1}$.
}
\end{remarks}

The rest of this section is devoted to seeking definitive relations
between and among the cardinals
$S((X_I)_\kappa)$, $S((X_J)_\kappa)$ with $J\in[I]^{<\kappa}$, and
$(\alpha^{<\kappa})^+$. Our success, though substantial, is only
partial, since we have been unable to give a fully satisfactory
answer to Question~\ref{Q1} in ZFC.

\begin{theorem}\label{T4.16}
Let $\alpha\geq2$ and $\kappa\geq\omega$ be cardinals such that
$\alpha^{<\kappa}<(\alpha^{<\kappa})^{<\kappa}$. If $\{X_i:i\in I\}$ is
a set of nonempty spaces such that
$S((X_I)_\kappa)>(\alpha^{<\kappa})^+$, then there are
a cardinal $\lambda<\kappa$ and $J\in[I]^{<\lambda}$ such that
$S((X_J)_{\lambda})\geq(\alpha^{<\kappa})^+$.
\end{theorem}
\begin{proof}
Let $\sC$ be a cellular family of basic open subsets of $(X_I)_\kappa$
such that $|\sC|=(\alpha^{<\kappa})^+$. Let
$\{\kappa_\eta:\eta<\cf(\kappa)\}$ be as in Notation~\ref{defkappaeta}
and for $\eta<\cf(\kappa)$ set
$\sC(\eta):=\{U\in\sC:|R(U)|<\kappa_\eta\}$.
Since $|\sC|=\bigcup_{\eta<\cf(\kappa)}\,\sC(\eta)$ and
$\cf((\alpha^{<\kappa})^+)=(\alpha^{<\kappa})^+\geq\kappa^+>\cf(\kappa)$,
there is $\eta<\cf(\kappa)$ (henceforth fixed) such that
$|\sC(\eta)|=(\alpha^{<\kappa})^+$. Then $\sC(\eta)$ is cellular in
$(X_I)_{\kappa_\eta}$,
hence in $(X_I)_{\kappa_\eta^+}$, and for $\eta<\eta'<\cf(\kappa)$ we
have
$$S((X_I)_{\kappa_\eta^+})>|\sC(\eta)|
=(\alpha^{<\kappa})^+>(\alpha^{\kappa_{\eta'}})^+.$$
Then since $\kappa_\eta^+\ll(\alpha^{\kappa_{\eta'}})^+$ there is, by
Theorem~\ref{CN} (with $\kappa_\eta^+$ and
$(\alpha^{\kappa_{\eta'}})^+$ in the roles of $\kappa$ and $\alpha$
respectively), a set $J(\eta')\in[I]^{<\kappa_\eta^+}$ such
that
$S((X_{J(\eta')})_{\kappa_\eta^+})>(\alpha^{\kappa_{\eta'}})^+$.
Then with $J:=\bigcup_{\eta<\eta'<\cf(\kappa)}\,J(\eta')$ we have
$|J|\leq\kappa_\eta\cdot\cf(\kappa)<\kappa$, and
$$S((X_J)_{\kappa_\eta^+})>(\alpha^{\kappa_{\eta'}})^+
 \mbox{ when }\eta<\eta'<\cf(\kappa),$$
hence $$S((X_J)_{\kappa_\eta^+})\ge\sup_{\eta<\eta'<\cf(\kappa)}(\alpha^{\kappa_{\eta'}})^+=\Sigma_{\eta<\eta'<\cf(\kappa)}\,
(\alpha^{\kappa_{\eta'}})^+=\alpha^{<\kappa}.$$ Since in our case 
$\alpha^{<\kappa}$ is singular (Theorem \ref{weak<kappa}(b)) we have 
$$S((X_J)_{\kappa_\eta^+})\geq(\alpha^{<\kappa})^+,$$ so the 
conclusion holds with $\lambda:=\max\{\kappa_\eta^+,|J|^+\}$.
\end{proof}
We continue in Corollary~\ref{GCNan} with a consequence of
Theorem~\ref{T4.16} for which Lemma~\ref{newl} is preparatory.
\begin{lemma}\label{newl}
Let $\kappa\ge\omega$ be a limit cardinal, $\alpha\geq2$ be a cardinal and 
$\{X_i:i\in I\}$ be a set of nonempty spaces. Then 
$S((X_J)_\lambda)<\alpha$ for 
each $\lambda<\kappa$ and each nonempty $J\in[I]^{<\lambda}$ if and only if 
$S((X_J)_\kappa)<\alpha$ for each nonempty $J\in[I]^{<\kappa}$.
\end{lemma}

\begin{proof}
Let $S((X_J)_\kappa)<\alpha$ for each nonempty
$J\in[I]^{<\kappa}$ and let $\lambda < \kappa$
and $\emptyset\neq J_0\in[I]^{<\lambda}$. Then
$S((X_{J_0})_\lambda)=S((X_{J_0})_\kappa)<\alpha$ since the 
spaces $(X_{J_0})_\lambda$ and $(X_{J_0})_\kappa$ have the full box topology 
and therefore coincide.

For the converse, let $S((X_J)_\lambda)<\alpha$ for 
each $\lambda<\kappa$ and each nonempty $J\in[I]^{<\lambda}$ and let 
$\emptyset\neq J_0\in[I]^{<\kappa}$. Since $|J_0|<\kappa$ and $\kappa$ is 
a limit cardinal, there exists $\lambda<\kappa$ such that $|J_0|<\lambda$.
Then $S((X_{J_0})_\kappa)=S((X_{J_0})_\lambda)<\alpha$ since the 
spaces $(X_{J_0})_\kappa$ and $(X_{J_0})_\lambda$ have the full box topology 
and therefore coincide.
\end{proof}

\begin{corollary}\label{GCNan}
Let $\alpha\geq2$ and $\kappa\ge\omega$ be cardinals
such that $\alpha^{<\kappa}<(\alpha^{<\kappa})^{<\kappa}$
and let $\{X_i:i\in I\}$ 
be a set  of nonempty spaces such that
$S((X_J)_\kappa)<(\alpha^{<\kappa})^+$ for each nonempty $J\in[I]^{<\kappa}$. 
Then $S((X_I)_\kappa)\le(\alpha^{<\kappa})^+$.
\end{corollary}
\begin{proof}
The cardinal $\kappa$ is singular by Theorem~\ref{weak<kappa}(b),
hence is a limit cardinal. Then by Lemma~\ref{newl}
(with $\alpha$ there replaced by $(\alpha^{<\kappa})^+)$ we have
$S((X_J)_\lambda)<(\alpha^{<\kappa})^+$ whenever $\lambda<\kappa$ and
$J\in[I]^{<\lambda}$. Then $S((X_I)_\kappa)\leq(\alpha^{<\kappa})^+$ by
Theorem~\ref{T4.16}.
\end{proof}

Theorem~\ref{GCNab}, like Corollary~\ref{S(power)}, is a 
miscellaneous stand-alone result based on the homeomorphisms developed in 
Lemma~\ref{homeos}. To see that those results are (consistently) nonvacuous, 
we need a model of ZFC where $\alpha^{<\kappa}<(\alpha^{<\kappa})^{<\kappa}$ 
and $\alpha^\kappa>(\alpha^{<\kappa})^+$. For that, see Remark~\ref{nonvac2}.
In \ref{homeos}--\ref{S(power)}, given a set $\{X_i:i\in I\}$ of spaces, for
$i\in I$ we write 
$$\widetilde{i}:=\{j\in I:X_i=_h X_j\}.$$
We note that if $\kappa\geq\omega$ and $|\widetilde{i}|\geq\cf(\kappa)$
for each $i\in I$, then there is a partition
$\{I(\eta):\eta<\cf(\kappa)\}$ of $I$ such that
$|I(\eta)\cap\widetilde{i}|=|\widetilde{i}|$ for each $i\in I$. Indeed,
it is enough for each $i\in I$ to choose a partition
$\{A(\widetilde{i},\eta):\eta<\cf(\kappa)\}$ of $\widetilde{i}$ with
each
$|A(\widetilde{i},\eta)|=|\widetilde{i}|$ and to take
$I(\eta):=\bigcup_{i\in I}\,A(\widetilde{i},\eta)$. 

\begin{lemma}\label{homeos}
Let $\kappa\geq\omega$ and $\cf(\kappa)\le \lambda\le \kappa$ with
$\lambda$ regular, and let $\{X_i:i\in I\}$ be a set of spaces with each
$|\widetilde{i}|\geq\cf(\kappa)$. Let $\{I(\eta):\eta<\cf(\kappa)\}$ be
a partition of $I$ such that $|I(\eta)\cap\widetilde{i}|=|\widetilde{i}|$
for each $i\in I$. Then

{\rm (a)} $(X_I)_\lambda=_h(X_{I(\eta)})_\lambda$ for each
$\eta<\cf(\kappa)$;

{\rm (b)}
$(X_I)_\lambda=_h(\Pi_{\eta<\cf(\kappa)}\,(X_{I(\eta)})_\lambda)_\lambda$;
and

{\rm (c)} $(X_I)_\lambda=_h(((X_I)_\lambda)^{\cf(\kappa)})_\lambda$.
\end{lemma}
\begin{proof} (a) Given $\eta<\cf(\kappa)$,
let $\phi:I\twoheadrightarrow I(\eta)$ be a bijection such
that $\phi[\widetilde{i}]=\widetilde{i}\cap I(\eta)$ for each $i\in I$.
Then the map $\Phi:X_I\twoheadrightarrow X_{I(\eta)}$ given by
$\Phi(x_i)=x_{\phi(i)}\in X_{I(\eta)}$ is a homeomorphism, with
$\phi[R(A)]=R(\Phi[A])$ for each generalized rectangle $A=\Pi_{i\in
I}\,A_i\subseteq X_I$.

(b) We show that the natural map from 
$\Pi_{\eta<\cf(\kappa)}\,X_{I(\eta)}$
onto $X_I$ is a homeomorphism from
$(\Pi_{\eta<\cf(\kappa)}\,(X_{I(\eta)})_\lambda)_\lambda$ onto
$(X_I)_\lambda$ when $\lambda$ is regular. Indeed, an (open) generalized
rectangle $U=\Pi_{\eta<\cf(\kappa)}\,U(\eta)$
in $\Pi_{\eta<\cf(\kappa)}\,X_{I(\eta)}$ with
$U(\eta)=\Pi_{i\in I(\eta)}\,U(\eta,i)$ satisfies
$R(U)=\bigcup_{\eta<\cf(\kappa)}\,R(U(\eta))$, so $|R(U)|<\lambda$
if and only if each $|R(U(\eta))|<\lambda$.

(c) follows immediately from (a) and (b).
\end{proof}

\begin{theorem}\label{GCNab}
Let $\alpha\geq2$ and $\kappa\ge\omega$ be cardinals
and let $\{X_i:i\in I\}$ 
be a set  of nonempty spaces.
Suppose that $\alpha^{<\kappa}<(\alpha^{<\kappa})^{<\kappa}$
and that $|\widetilde{i}|\geq\cf(\kappa)$ for each $i\in I$. If
$S((X_I)_\kappa)>(\alpha^{<\kappa})^+$ then
$S((X_I)_\kappa)>\alpha^\kappa$.
\end{theorem}

\begin{proof}
By Theorem~\ref{T4.16} there are $J\in[I]^{<\kappa}$ and a regular
cardinal $\lambda<\kappa$ such that
$S((X_J)_\lambda)\geq(\alpha^{<\kappa})^+$. Let
$\{I(\eta):\eta<\cf(\kappa)\}$ be a partition of $I$ as in
Lemma~\ref{homeos}, and for $\eta<\cf(\kappa)$ let $\sC(\eta)$ be a
cellular family in $X_{I(\eta)}$ such that
$|\sC(\eta)|=\alpha^{<\kappa}$. Then
$$\sC:=\Pi_{\eta<\cf(\kappa)}\,\sC(\eta)=\{\Pi_{\eta<\cf(\kappa)}\,C(\eta):C(\eta)\in\sC(\eta)\}$$
is cellular in 
$(\Pi_{\eta<\cf(\kappa)}\,(X_{I(\eta)})_\lambda)_\lambda=_h(X_I)_\lambda$,
so
$$S((X_I)_\kappa)\geq
S((X_I)_\lambda)>|\sC|=(\alpha^{<\kappa})^{\cf(\kappa)}=\alpha^\kappa.$$
\vskip-18pt
\end{proof}

\begin{corollary}\label{S(power)}
Let $\alpha\geq2$ and $\kappa\geq\omega$ be cardinals, and let $X$ be a
space and $I$ a set.

{\rm (a)} Suppose that $\alpha^{<\kappa}=(\alpha^{<\kappa})^{<\kappa}$.
Then $S((X^I)_\kappa)\leq(\alpha^{<\kappa})^+$ if and only if
$S((X^J)_\kappa)\leq(\alpha^{<\kappa})^+$ for every nonempty
$J\in[I]^{<\kappa}$.

{\rm (b)} Suppose that $\alpha^{<\kappa}<(\alpha^{<\kappa})^{<\kappa}$.
If $S((X^I)_\kappa)>(\alpha^{<\kappa})^+$ then

\begin{itemize}
\item[{\rm (1)}] there is nonempty $J\in[I]^{<\kappa}$ such that
$S((X^J)_\kappa)\geq(\alpha^{<\kappa})^+$; and
\item[{\rm (2)}] if $|I|\geq\cf(\kappa)$, then $S((X^I)_\kappa)>\alpha^\kappa$.
\end{itemize}
\end{corollary}

\begin{proof} In view of Theorem~\ref{GCN} and Theorem~\ref{T4.16}, only (b)(2) 
requires attention. This follows from Theorem~\ref{GCNab}, since now
$\widetilde{i}=I$ for each $i\in I$.
\end{proof}

\begin{remarks}\label{nonvac2}
{\rm
(a) It is easy to see that in many models of ZFC, for example under GCH, the
equality $\alpha^\kappa=(\alpha^{<\kappa})^+$ holds for all cardinals
$\alpha$ and $\kappa$ for which
$\alpha^{<\kappa}<(\alpha^{<\kappa})^{<\kappa}$. In such models,
Theorem~\ref{GCNab} and Corollary~\ref{S(power)}(b)(2) become
tautologies. To see that Theorem~\ref{GCNab} and 
Corollary~\ref{S(power)}(b)(2) are not vacuous in every setting, it is enough 
to refer to the models $\VV_1$ and $\VV_2$ of Gitik and Shelah described in
Discussion~\ref{gitik-shelah}(d), taking now $\alpha=2$ and
$\kappa=\aleph_\omega$. In those models we have
$$\alpha^{<\kappa}=2^{<\aleph_\omega}=\aleph_\omega,$$
while (using Theorem~\ref{expon}(c), for example)
$$\alpha^\kappa=(\alpha^{<\kappa})^{<\kappa}
=2^{\aleph_\omega}=\aleph_{\omega+2}>\aleph_{\omega+1}=(\alpha^{<\kappa})^+.$$

(b) It is a consequence of Theorem~\ref{GCNab} and 
Corollary~\ref{S(power)}(b)(2) that under the hypotheses
there the relation $S((X_I)_\kappa)=\alpha^\kappa$ is impossible (even
when $\alpha^\kappa$ is regular). 
}
\end{remarks}

Now we consider two questions. The first of these arises naturally
from Corollary~\ref{GCN}(b) and Corollary \ref{GCNan}, and a version of the 
second, attributed to Argyros and Negrepontis, appears in \cite{comfneg82}. 
Theorem~\ref{GCH} shows a relation between these. For what we do and do not not 
know about the status of these questions in ZFC and in augmented systems, see 
Remarks \ref{R4.50}((a) and (b)).

\begin{question}\label{Q1}
{\rm
Let $\alpha\geq2$, $\kappa\geq\omega$, and let $\{X_i:i\in I\}$ be a
set of spaces such that $S((X_J)_\kappa)\leq(\alpha^{<\kappa})^+$ for
each nonempty $J\in[I]^{<\kappa}$. Is then necessarily
$S((X_I)_\kappa)\leq(\alpha^{<\kappa})^+$?
}
\end{question}

\begin{question}[{\cite[7.15(a)]{comfneg82}}]\label{Q}
{\rm
Are there spaces $X$ and $Y$ with $S(X\times Y)>S(X)>S(Y)$?
}
\end{question}

\begin{remark}\label{Qa}
{\rm
By Theorem~\ref{GCN}(b), the answer to Question~\ref{Q1} is
affirmative in case $\alpha^{<\kappa}=(\alpha^{<\kappa})^{<\kappa}$.
}
\end{remark}

\begin{theorem}\label{GCH}
Let $\alpha\geq2$, $\kappa\geq\omega$, and $\{X_i:i\in I\}$ 
witness a negative answer to Question \ref{Q1}. 
If $\alpha<\kappa$ then the answer to Question~\ref{Q} is positive.
\end{theorem}

\begin{proof}
Set
\begin{center}
$L:=\{i\in I:S(X_i)=(\alpha^{<\kappa})^+\}$ and
$M:=\{i\in I:S(X_i)<(\alpha^{<\kappa})^+\}$.
\end{center}

Note first from Remark \ref{Qa} that 
$\alpha^{<\kappa}<(\alpha^{<\kappa})^{<\kappa}$, so $\kappa>\omega$, and 
$\kappa$ and $\alpha^{<\kappa}$ are singular by 
Theorem~\ref{weak<kappa}. Further,
it follows directly from 
our hypothesis that $|I|\ge\kappa$. Clearly we may assume
without loss of
generality that $S(X_i)\ge 3$ for every $i\in I$.

For each infinite cardinal $\lambda$ we have
$S((X_I)_{\lambda^+})=S((X_L)_{\lambda^+}\times(X_M)_{\lambda^+})$. Thus
to prove the theorem it suffices to show that there exists
an infinite cardinal  $\lambda<\kappa$ such that
\begin{itemize}
\item[(i)] $S((X_L)_{\lambda^+})=(\alpha^{<\kappa})^+$,
\item[(ii)] $S((X_M)_{\lambda^+})<\alpha^{<\kappa}$, and
\item[(iii)] $S((X_I)_{\lambda^+})>(\alpha^{<\kappa})^+.$
\end{itemize}

Let $\{\kappa_\eta:\eta<\cf(\kappa)\}$ be a family 
of cardinals as in Notation \ref{defkappaeta}.

We note that
\begin{equation}\label{Eq15}
|L|<\cf(\kappa),
\end{equation}
and
\begin{equation}\label{Eq16}
\mbox{there is }\eta<\cf(\kappa)\mbox{ such that }
S(X_i)\leq\alpha^{\kappa_\eta} \mbox{ for each }i\in M.
\end{equation}
\noindent (Indeed if (\ref{Eq15}) [resp., (\ref{Eq16})] fails then by 
Theorem~\ref{largeS}(b) 
there is $J\in[L]^{\cf(\kappa)}$ [resp., $J\in[M]^{\cf(\kappa)}$] such that
$$S((X_J)_\kappa)\geq S((X_J)_{(\cf(\kappa))^+})>(\alpha^{<\kappa})^+,$$
a contradiction since $|J|=\cf(\kappa)<\kappa$.)

It follows from (\ref{Eq15}) that $|M|=|I|\ge \kappa$; further,  
according to Lemma \ref{obs1} (with $M$, $\gamma^+$, $2$ and
$(\alpha^{<\kappa})^+$ in place of $I$, $\kappa$, $\beta$ and $\alpha$,
respectively), we have 
\begin{equation}\label{Eq14}
S((X_M)_{\gamma^+})> 2^\gamma \mbox{ for every infinite }\gamma<\kappa.
\end{equation}
We claim that 
\begin{equation}\label{Eq17}
S((X_L)_\kappa)=S((X_M)_\kappa)=(\alpha^{<\kappa})^+.
\end{equation}
To see that, fix $\eta$ as in (\ref{Eq16}) and let $\gamma$ be such that
$\kappa_\eta<\gamma<\kappa$. Since $\alpha<\kappa$, we have
$(\alpha^\gamma)^+<\alpha^{<\kappa}$;
then 
\begin{equation}\label{Eq18}
S((X_M)_{\gamma^+})\leq(((\alpha^{\kappa_\eta})^\gamma)^\gamma)^+=
(\alpha^\gamma)^+<\alpha^{<\kappa}
\end{equation}
by (\ref{Eq16}) and Theorem~\ref{GKurepa}(b) (with $\alpha$, $\kappa$ and $I$ 
replaced by $\alpha^{\kappa_\eta}$, $\gamma^+$ and $M$, respectively). 
It then follows from (\ref{Eq14}) and (\ref{Eq18}) that
$$\sup\{S((X_M)_\gamma):\gamma<\kappa\}=\alpha^{<\kappa},$$ 
hence from Theorem \ref{Obsn}(c) we have 
$$S((X_M)_{\kappa})=(\alpha^{<\kappa})^+.$$
Since
$S((X_M)_\kappa)=(\alpha^{<\kappa})^+<S((X_I)_\kappa)$
we have $M\neq I$, so $L\neq\emptyset$. Clearly then
$S((X_L)_\kappa)\geq(\alpha^{<\kappa})^+$, 
while $S((X_L)_\kappa)\leq(\alpha^{<\kappa})^+$
follows from (\ref{Eq15}) and our hypothesis.
Therefore $S((X_L)_\kappa)=(\alpha^{<\kappa})^+$ 
and claim (\ref{Eq17}) is proved.

From (\ref{Eq17}) we have
$$(\alpha^{<\kappa})^+\leq
S((X_L)_{\lambda^+})\leq S((X_L)_\kappa)=(\alpha^{<\kappa})^+$$
for each infinite $\lambda$, so (i) holds (for all infinite
$\lambda$). That (ii) holds for all $\lambda$ such that
$\kappa_\eta<\lambda<\kappa$ is given by (\ref{Eq18}).
It follows that there is $\lambda$ such that 
$\kappa_\eta<\lambda<\kappa$ and
$S((X_I)_{\lambda^+})>(\alpha^{<\kappa})^+$, since otherwise we have
$$\sup\{S((X_I)_{\lambda^+}):\lambda<\kappa\}=\sup\{S((X_I)_\lambda):\lambda<\kappa\}=(\alpha^{<\kappa})^+$$
and Theorem \ref{Obsn}(b) gives the contradiction 
$S((X_I)_\kappa)=(\alpha^{<\kappa})^+$. Then (iii) holds for that
specific $\lambda$.
\end{proof}

\begin{corollary}\label{CGCH}
Let $\MM$ be a model of ZFC in which every singular cardinal is a strong 
limit cardinal (e.g. $\MM$ is a model of ZFC+GCH). If the answer to 
Question \ref{Q1} is negative in $\MM$ then the answer to Question~\ref{Q} is 
positive in $\MM$.
\end{corollary}

\begin{proof}
With $\alpha$, $\kappa$ and $\{X_i:i\in I\}$ chosen as in Theorem \ref{GCH}
it suffices to show that $\alpha<\kappa$. 

Since $\alpha^{<\kappa}<(\alpha^{<\kappa})^{<\kappa}$ by Remark \ref{Qa},
by Theorem \ref{weak<kappa}(b), both $\kappa$ and 
$\alpha^{<\kappa}$ are singular. If $\alpha=\alpha^{<\kappa}$ then
$\alpha=\alpha^{<\kappa}=(\alpha^{<\kappa})^{<\kappa}$, a 
contradiction; therefore $\alpha<\alpha^{<\kappa}$. Suppose now that
$\kappa\leq\alpha$. Then from
Theorem~\ref{expon}(c) we have 
$(\alpha^{<\kappa})^{<\kappa}=\alpha^\kappa\le\alpha^\alpha=2^\alpha$, and the
relation $\alpha<\alpha^{<\kappa}<2^\alpha$ contradicts the hypothesis
that $\alpha^{<\kappa}$ is a strong limit cardinal.
\end{proof}

\begin{remarks}\label{R4.50}
{\rm (a) The proof of the previous corollary does not need the full hypothesis 
that every singular cardinal in $\MM$ is strong limit. It is enough to know just 
that $\alpha^{<\kappa}$ is strong limit.

(b) ZFC-consistent examples of spaces as requested in Question~\ref{Q} are 
available in the literature. 

~~(1) In the Cohen models of Fleissner \cite{fleissner78} (see Remark~\ref{G-Smodels}(b))
there are spaces $X$ and $Y$ such that 
$$S(Y\times Y)=\aleph_{\omega+2}>\cc=\aleph_{\omega+1}>\aleph_1=S(Y),$$ 
and then with $X$ the ``disjoint union" of $D(\aleph_1)$ and $Y$ we have 
$S(X\times Y)>S(X)>S(Y)$.

~~(2) It is shown by Shelah \cite[4.4]{shelah} that if $\kappa$ is a singular strong limit 
cardinal such that $\lambda:=\kappa^+=2^\kappa$ then there are spaces $X$ and $Y$ such that 
$$S(X\times Y)\ge \lambda^{++}>\lambda^+=S(X)>\lambda>\kappa>(2^{\cf(\kappa)})^{++}\ge S(Y).$$

(c) We do not know if there are models of ZFC in which no spaces as in 
Question~\ref{Q} exist. We do not know if the answer to Question~\ref{Q1} is 
absolutely or consistently ``Yes", absolutely or consistently ``No".
}
\end{remarks}

\end{document}